\documentclass[final,onefignum,onetabnum]{siamart190516}
\usepackage{amssymb}
\usepackage{graphicx} 
\usepackage{xspace}
\usepackage{multirow}
\usepackage{Macros/mydef}
\usepackage{Macros/remarks}
\usepackage{algorithm}
\usepackage{algpseudocode}
\usepackage{subcaption}
\usepackage{enumitem}
\usepackage{hhline}
\newlength{\Oldarrayrulewidth}

\begin{document}

\title{Numerical Simulations of Surface-Quasi Geostrophic Flows on Periodic Domains
\thanks{\funding{A.B. is partially supported by the NSF Grant DMS-1817691; M.N. is partially supported by Esseen scholarship at Uppsala University.}}}

\author{Andrea Bonito\thanks{Department of Mathematics, Texas A\&M University, College Station, TX 77843.}
\and Murtazo Nazarov \thanks{Corresponding author. Department of Information Technology, Uppsala University, SE 75105.}}

\maketitle

\begin{abstract}
We propose a novel algorithm for the approximation of surface-quasi geostrophic (SQG) flows modeled by a nonlinear partial differential equation coupling transport and fractional diffusion phenomena. The time discretization consists of an explicit strong-stability-preserving three-stage Runge-Kutta method while a flux-corrected-transport (FCT) method coupled with Dunford-Taylor representations of fractional operators is advocated for the space discretization. Standard continuous piecewise linear finite elements are employed and the algorithm does not have restrictions on the mesh structure nor on the computational domain. In the inviscid case, we show that the resulting scheme satisfies a discrete maximum principle property under a standard CFL condition and observe, in practice, its second-order accuracy in space. The algorithm successfully approximates several benchmarks with sharp transitions and fine structures typical of SQG flows. In addition, theoretical Kolmogorov energy decay rates are observed on a freely decaying atmospheric turbulence simulation.
\end{abstract}



\pagestyle{myheadings} \thispagestyle{plain} \markboth{ANDREA BONITO, MURTAZO NAZAROV}{Numerical simulations of SQG flows}

\section{Introduction}
\label{sec:intro}

The Navier-Stokes system models the behavior of  incompressible, adiabatic, inviscid fluids in hydrostatic balance. 
When, in addition, the fluid is constrained by environmental rotation and stratification, Charney \cite{charney1990scale} derived in the 1940's a three dimensional quasi-geostrophic model to describe large-scale mid-latitude atmospheric motions and oceanographic motions. \! Charney's quasi-geostrophic model received much attention, we mention \cite{gill2016atmosphere,majda2003introduction,pedlosky2013geophysical,lacasce2006estimating,lapeyre2006dynamics,vallis2017atmospheric} for discussions on its validity.

In a surface quasi-geostrophic (SQG) setting, it is further assumed that the potential vorticity is uniform, see for instance \cite{pedlosky2013geophysical,charney1971geostrophic,lapeyre2017surface,constantin1994SQG,constantin1998SQG,held1995surface}. 
Consequently, on the half plane above the surface $\mathcal S:=\{ (x_1,x_2,x_3) \in \mathbb R^3 \ : \ x_3=0\}$, the vorticity $\widetilde \psi(x_1,x_2,x_3,t)$ satisfies
\begin{equation}\label{e:laplacianwhole}
\Delta \widetilde \psi = 0, \quad \textrm{where } x_3>0 \quad \textrm{and} \quad \lim_{z\to \infty} \widetilde\psi(x_1,x_2,x_3,t) = 0.
\end{equation}
On $\mathcal S$,  the buoyancy (or potential temperature) is given by $\theta := \p_{x_3} \widetilde \psi|_{x_3=0}$.

We restrict our considerations to surface consisting in a rectangular periodic domain $\Omega:=(0,\pi)^2$  denoted $\mathbb T^2$ in short. When restricted to $\mathbb T^2$, \eqref{e:laplacianwhole} corresponds to a nonlocal elliptic partial differential equation involving $\theta(x_1,x_2,t)$ and $\psi(x_1,x_2,t):=\widetilde \psi(x_1,x_2,0,t)$, namely
$$
(-\Delta)^{\frac 1 2}\psi = \theta,
$$
where $(-\Delta)^{\frac 1 2}$ stands for the \emph{spectral} fractional laplacian defined in Section~\ref{ss:FracLap}.
The buoyancy is transported on $\mathbb T^2$ along the orthogonal directions to the vorticity gradient $\mathbf u := \nabla^\perp \psi $, where for $v:\mathbb R^2 \rightarrow \mathbb R$ we set  \mbox{$
\nabla^\perp v := 
\begin{pmatrix}
-\p_{x_2} v\\
~\p_{x_1} v
\end{pmatrix}$.} 
In addition, we account for the Ekman pumping effect (friction between vertical thin layers of atmosphere) resulting in the following nonlinear advection-diffusion relation for the buoyancy 
$$
 \p_t \theta + \bu\SCAL\GRAD  \theta  + \varkappa (-\Delta)^{\frac 1 2} \theta = 0,
$$
where  $\varkappa \geq 0$ stands for the Ekman pumping coefficient. 
This coefficient is typically small except on narrow boundary layers touching the fluid boundary \cite{pedlosky2013geophysical}.  
The case $\varkappa =0$ will be referred to as the inviscid case.  A detailed derivation of the SQG system can be found in \cite{penghesis}, based on the works \cite{pedlosky2013geophysical,charney1971geostrophic,held1995surface,leith1980nonlinear,hoskins1975geostrophic} and \cite{bonito2020electroconvection,stinga2010extension}. 
	
The above nonlinear system of equations features many aspect of large-scale atmospheric motions. 
Among them, we list the apparition in finite time of discontinuous temperature - called Frontogenesis - and the conservation (for $\varkappa=0$) of the kinetic energy and helicity, see Section~\ref{sec:numerical_tests}.
Whether solutions to the SQG equations can develop singularities is a question which concerned many researchers and global regularity for general data remains an open problem \cite{constantin2012SQGsimu}. 
We refer to \cite{caffarelli2010drift,constantin2001critical,constantin1994SQG,resnick1996dynamical,kiselev2007global,buckmaster2019nonuniqueness} for additional information.

Numerical methods are of fundamental importance to assess the behavior of the solutions to the SQG system.
Particular attention must be made to reproduce accurately discontinuous profiles while conserving the kinetic energy and helicity.
%:
Existing numerical algorithms for the approximation of the SQG system are based on spectral decompositions of the solution coupled with higher order exponential filters, see \cite{constantin2012SQGsimu,constantin1994SQG,constantin1998SQG} and \cite{Karniadakis2017SQG}.
Instead, our approach is based on standard finite element discretization with nonlinear stabilization. 

We summaries in Section~\ref{sec:prelim} the SQG system along with some of its important properties. 
In Section~\ref{s:numerical_alg}, we propose to adapt the algorithms proposed in \cite{bonito2015numerical,bonito2017numerical,bonito2019numerical,bonito2019sinc} to the present periodic setting and employ a flux corrected transport (FCT) limiting blending a low order scheme satisfying a discrete maximum principle (when $\varkappa =0$) with a higher order shock-capturing method \cite{Guermond_Nazarov_Popov_Yong_2014, Guermond_Popov_2017,Guermond_et_al_2018}.
The resulting scheme retain the maximum preserving property and is observed in practice to retain the higher order accuracy.
At this point it is worth mentioning that there seem to be no mathematical explanation of the higher order properties of FCT algorithm available in the literature. 
We showcase in Section~\ref{sec:numerical_tests} the need of the FCT algorithm to avoid over-diffusive simulations. 
In fact, the numerical simulations obtained exhibit sharp resolutions of line discontinuities and fine structures.
In addition, we propose numerical simulations of freely decaying turbulence and confirm the predictions of \cite{held1995surface,Tulloch14690,lapeyre2017surface,Ragone_Badin_2016,capet2008surface} for the decay of the kinetic energy cascade for the inviscid $(\varkappa=0)$ and diffuse $(\varkappa>0)$ SQG system. At large scales we recover the $-\frac53$ Kolmogorov rate of decay  typical to three dimensional flows while  a $-3$  Kolmogorov rate of decay, this time typical of two dimensional flows, is observed at small scales.

\section{Preliminaries}\label{sec:prelim}

\subsection{The Spectral Fractional Laplacian on the Torus}\label{ss:FracLap}

To define the fractional laplacian in \eqref{eq:sqg_frac}, we denote by $\{(\lambda_i,\phi_i)\}_{i=0}^\infty \subset \mathbb R_+ \times H^1(\mathbb T^2)$ the eigenpairs of the Laplacian on the torus $\mathbb T^2$.  
We use the convention $0=\lambda_0<\lambda_1\leq \lambda_2 \leq ...$ and assume that the $\phi_i$'s are orthonormal in $L^2(\mathbb T^2)$ and orthogonal in $H^1(\mathbb T^2)$. 

For $-1\leq r \leq 1$, the fractional power of the Laplacian is defined for smooth functions $v \in C^\infty(\mathbb T^2)$ with vanishing mean value as
\begin{equation}\label{e:spectral_lap}
(-\Delta)^r v := \sum_{n=1}^\infty \lambda_i^r v_i \phi_i, \qquad v_i := \int_{\mathbb T^2} v(\bx) \phi_i(\bx)d\bx. 
\end{equation}

The definition of the fractional laplacian \eqref{e:spectral_lap} is extended by density to 
$$
\mathcal D((-\Delta)^s) := \left\lbrace v \in L^2_\#(\mathbb T^2) \ : \ \sum_{i=0}^\infty \left(\int_{\mathbb T^2} v \phi_i \right)^2 \lambda_i^{2s} < \infty \right\rbrace,
$$
where $L^2_\#(\mathbb T^2)$ is the subspace of $L^2(\mathbb T^2)$ consisting of vanishing mean value functions. 

For latter use, we record the following relation directly following from the definition of the fractional laplacian 
\begin{equation}\label{e:int_parts}
\int_{\mathbb T^2} (-\Delta)^{s_1} v (-\Delta)^{s_2} w  = \int_{\mathbb T^2} (-\Delta)^{r_1} v (-\Delta)^{r_2} w, \qquad  v,w \in C^\infty(\mathbb T^2) \cap L^2_\#(\mathbb T^2),
\end{equation}
for $-1 \leq s_1 \leq r_1 \leq r_2 \leq s_2\leq 1$ satisfying $s_1+s_2 = r_1+r_2$.

\subsection{The SQG Equations}

We denote by $T$ the final time.
The solution to the SQG system is a pair $\theta,\psi: \mathbb T^2 \times [0,T] \rightarrow \mathbb R$ satisfying  
\begin{equation}\label{eq:sqg_temp}
		\p_t \theta + \bu\SCAL\GRAD \theta + \varkappa (-\Delta)^{s} \theta = 0, \qquad \textrm{in }\mathbb T^2 \times (0,T]		
	\end{equation}
and
\begin{equation}\label{eq:sqg_frac}
		(-\Delta)^{\frac 1 2}\psi = \theta, \quad \mathbf u = \nabla^\perp \psi \qquad \textrm{in }\mathbb T^2 \times (0,T].
\end{equation}
Here we introduced a parameter $0< s<1$ to include additional mathematical models considered in the literature for the design of the numerical method.
However, our numerical experiments focus on the physical SQG system and thus on the critical case $s=\frac 12$.
The system of equations \eqref{eq:sqg_temp} and \eqref{eq:sqg_frac} is supplemented by the initial and mean value conditions
$$
\theta(.,0) = \theta_0 \quad \textrm{in }\mathbb T^2, \qquad \int_{\mathbb T^2} \theta =  \int_{\mathbb T^2} \psi = 0 \qquad \textrm{in }(0,T),
$$where $\theta_0: \mathbb T^2 \rightarrow \mathbb R$ is a given initial buoyancy satisfying $\int_{\mathbb T^2} \theta_0(\bx)= 0$.

From now on we assume that there exists a unique sufficiently smooth solution $(\theta,\psi)$ and refer to works cited in the introduction for discussions on the existence and uniqueness of solutions as well as their regularity.

\subsection{Kinetic Energy and Helicity}

The kinetic energy $\calK(\theta)$ and helicity $\calH(\theta)$ are defined by
\begin{equation}\label{eq:kinetic_helicity}
	\calK(\theta):= \frac12 \int_{\mathbb T^2} \theta^2(\bx,t) \ud \bx \quad  \mbox{ and } \quad 
	\calH(\theta):= -\!\!\int_{\mathbb T^2} \psi(\bx,t) \theta(\bx,t) \ud \bx
\end{equation}
and are monitored in several numerical experiments in Section~\ref{sec:numerical_tests} to showcase the performances of the proposed algorithm.
We also compute in Section~\ref{sec:2d_turbulence} the Kolmogorov energy cascades for the Kinetic model and validate our turbulence model.
Both are conserved quantities when $\varkappa =0$ and dissipated when $\varkappa>0$.
We make this more precise now.

To obtain an evolution relation for the kinetic energy, we multiply \eqref{eq:sqg_temp} by $\theta$ and integrate over $\mathbb T^2$ to get
\begin{equation}\label{eq:cons_kinetic_A}
\frac{d}{dt} \calK(\theta) = - \varkappa \int_{\mathbb T^2} (-\Delta)^{s} \theta ~\theta,
\end{equation}
where we used the definition $\bu = \nabla^\perp \psi$ to deduce that $\textrm{div}\bu = 0$ and so $
\int_{\mathbb T^2} \bu\SCAL\GRAD \theta~  \theta =0$.
The integration by parts relation~\eqref{e:int_parts} applied to the right hand side of \eqref{eq:cons_kinetic_A} yields
\begin{equation}\label{eq:cons_kinetic}
\frac{d}{dt} \calK(\theta) = - \varkappa \int_{\mathbb T^2} |(-\Delta)^{\frac s 2} \theta|^2.
\end{equation}

%From the definition  $\bu = \nabla^\perp \psi$, we directly get $\textrm{div}\bu = 0$.
%This in conjunction with \eqref{eq:sqg_temp} yields
%\begin{equation}\label{eq:cons_kinetic}
%\frac{d}{dt} \calK(\theta) = - \varkappa \int_{\mathbb T^2} |(-\Delta)^{\frac s 2} \theta|^2
%\end{equation}
%for sufficiently smooth function $\theta$.

We now turn our attention to the helicity.
We multiply \eqref{eq:sqg_temp} by $\psi$, integrate over $\mathbb T^2$ and invoke the relation $\psi = (-\Delta)^{-\frac 1 2} \theta$ to write
\begin{equation}\label{e:hel_mid}
\int_{\mathbb T^2} \partial_t \theta (-\Delta)^{-\frac 1 2}\theta + \int_{\mathbb T^2} \bu\SCAL\GRAD \theta \ \psi = - \varkappa \int_{\mathbb T^2} (-\Delta)^{s} \theta (-\Delta)^{-\frac 1 2} \theta.
\end{equation}
We rewrite the above three terms separately.
The term involving the velocity $\bu = \nabla^\perp \psi$ vanishes in this case as well
$$
\int_{\mathbb T^2} \bu \cdot \nabla \theta \psi = - \int_{\mathbb T^2} \bu \cdot \nabla \psi \theta = - \int_{\mathbb T^2} \nabla^\perp \psi \cdot \nabla \psi \theta = 0.
$$
For the left most term in \eqref{e:hel_mid}, we invoke the integration by parts formula \eqref{e:int_parts} (twice) and the relation $\psi = (-\Delta)^{-\frac 1 2} \theta$ to deduce
\begin{equation*}
\int_{\mathbb T^2} \partial_t \theta (-\Delta)^{-\frac 1 2}\theta = \frac 1 2 \frac{d}{dt}\int_{\mathbb T^2} |(-\Delta)^{-\frac 14} \theta|^2 =
 \frac 1 2 \frac{d}{dt}\int_{\mathbb T^2} \psi \theta =  -\frac 1 2 \frac{d}{dt} \mathcal H(\theta).
\end{equation*}
Using \eqref{e:int_parts} once again for the right hand side of  \eqref{e:hel_mid} yields
$$
- \varkappa \int_{\mathbb T^2} (-\Delta)^{s} \theta (-\Delta)^{-\frac 1 2} \theta = - \varkappa \int_{\mathbb T^2} |(-\Delta)^{\frac 1 2(s-\frac 1 2)} \theta|^2.
$$
Gathering the above relations, we obtain 
\begin{equation}\label{eq:cons_helicity}
\frac 1 2 \frac{d}{dt} \mathcal H(\theta) =   \varkappa \int_{\mathbb T^2} |(-\Delta)^{\frac 1 2(s-\frac 1 2)} \theta|^2.
\end{equation}

\section{Numerical Algorithm}\label{s:numerical_alg}

\subsection{The Finite Element Spaces}\label{sec:notations}

We propose to use continuous piecewise linear finite elements for the space approximation of the potential temperature $\theta$ and stream function $\psi$.
Let $\{ \calT_h \}_{h>0}$ be a sequence of shape-regular, quasi-uniform and conforming triangulations of $\mathbb T^2$ in the sense of \cite{ciarlet2002finite}, where $h:=\min_{K \in \mathcal T_h}\textrm{diam}(K)$ stands for the smallest diameter  of all the triangles in $\mathcal T_h$. 

To each triangulation $\calT_h$, we associate the spaces of continuous piecewise polynomial
\begin{equation}\label{e:FEM_space}
\calX_h:=\{ v_h\in \calC^0(\mathbb T^2);\, \forall K\in  \calT_h,\, v_h|_K \in \polP_1 \},  \quad \calX_{h,0}:= \calX_h \cap L^2_\#(\mathbb T^2),
\end{equation}
where $\polP_1$ denotes the space of polynomials of degree at most one and $\calC^0(\mathbb T^2)$ the space of continuous functions on $\mathbb T^2$ (and therefore $2\pi$-periodic on each variable).
We denote by $\{\varphi_1,\ldots,\varphi_I\}$ the basis of $\calX_h$ made of linear Lagrange finite elements (hat functions) associated with the collection of all the vertices $\{ \bx_j \}_{j=1}^I$ in the triangulation $\calT_h$ (not counting twice the periodic nodes).
The index list of basis functions interacting with  $\varphi_i$, $1\leq i \leq I$, is denoted by
\begin{equation}\label{e:index}
\calI(i):=\{ j \in \{1,...,I\} :  \textrm{supp}(\varphi_i) \cap \textrm{supp}(\varphi_j) \not = \emptyset \}.
\end{equation}
A mass lumping strategy detailed below will be critical to obtain maximum principle preserving schemes.
We denote by 
\begin{equation}\label{e:m}
m_{ij}:= \int_{\mathbb T^2} \varphi_j \varphi_i \quad \textrm{and} \quad m_i:= \sum_{j\in \calI(i)} m_{ij} = \int_{\mathbb T^2} \varphi_i
\end{equation}
the elements of the consistent and lumped mass matrices.

To ease the notations, we will use capital letters to denote finite element approximations and drop the subindex $h$. 
For instance, $\Theta \in \calX_{h,0}$ will denote the approximation of $\theta$.  

\subsection{Approximations of the Fractional Laplacian with Periodic Boundary Conditions}
\label{sec:fractional}

Several approaches are available for the approximation of the spectral fractional Laplacian.
We refer for instance the the reviews \cite{bonito2018numerical} and \cite{lischke2018fractional}. 
In this work, we adapt the algorithms developed in \cite{bonito2015numerical,bonito2017numerical,bonito2019sinc},  which are based on different Balakrshian-Dunford-Taylor representations described now.
We emphasis that the resulting algorithms consist of the agglomerations of  solutions to  advection-diffusion problems approximated using a standard continuous piecewise linear finite element space $\calX_h$. Their implementations are therefore straightforward and readily available in standard finite element softwares. Also, the algorithms presented do not suffer any restriction regarding the shape of the computational domain.

\subsubsection{Approximations of Negative Powers of Fractional Operators}
For $f \in L^2_\#(\mathbb T^2)$ and $s \in (0,1)$, we have the following representation
$$
(-\Delta)^{-s} f =  v:=\frac{1}{\pi} \int_{-\infty}^\infty e^{(1-s)y} w(y) dy,
$$
where $w(y)\in H^1(\mathbb T^2)\cap L^2_\#(\mathbb T^2)$ solves
$$
e^y w(y) -\Delta w(y) = f \qquad \textrm{in }\mathbb T^2,
$$
see e.g. \cite{MR1336382}.

A sinc quadrature is advocated for the approximation of the integral in $y$, thereby requiring the values of $w(y_\ell)$ at some selected snapshots $y_\ell \in \mathbb R$. The latter are approximated using a standard finite element method for reaction-diffusion problems.
Given a spacing parameter $k>0$ and integer $M\sim k^{-2}$,  we have  
\begin{equation}\label{e:appx_inv}
(-\Delta)^{-s} f \approx V_k := \frac 1 \pi k \sum_{\ell=-M}^{M} e^{(1-s)y_{\ell}}W(y^{\ell}),
\end{equation}
where $y_{\ell}:=\ell k$, $\ell=-M,...,M$, and $\calX_h \ni W(y^\ell) \approx w(y_\ell)$ solves
$$
e^{y_{\ell}} \int_{\mathbb T^2} W(y_{\ell}) R  + \int_{\mathbb T^2} \nabla W(y_{\ell}) \cdot \nabla R = \int_{\mathbb T^2} f R, \qquad \forall R \in \calX_h.
$$
Notice that when $\int_{\mathbb T^2} f  = 0$, we automatically have $\int_{\mathbb T^2} W(y_\ell) = 0$ and thus $V_k \in \calX_{h,0}$.
We refer to \cite{bonito2015numerical,bonito2019sinc} for the convergence analysis of $v_h^k$ towards $v$. We only  point out here that the convergence is exponential in $-1/k$ and optimal in $h$ (depending on the regularity of $f$ and the metric used to measure the error).

\subsubsection{Approximations of Positive Powers of Fractional Operators}{\textcolor{white}.}
While \eqref{e:appx_inv} is sufficient to design a numerical scheme approximating \eqref{eq:sqg_frac}, the explicit nature of our proposed time stepping scheme (see Section~\ref{s:numerical_alg}) also requires, when $\varkappa>0$, an approximation of 
$$
\int_{\mathbb T^2} (-\Delta)^s V\ W,
$$
for $\frac 1 2 \leq s<1$ and $V, W \in \calX_h$.
This time, we use the representation derived in \cite{bonito2019numerical} 
\begin{equation}\label{e:rep_op}
\int_{{\mathbb T^2}} (-\Delta)^s V \ W = 2\frac{\sin( \pi s)}{\pi}  \int_0^\infty e^{sy} \int_{{\mathbb T^2}} (V+\tilde V(y;V)) W  dy,
\end{equation}
which is valid for $0\leq s \leq 1$, $V,W \in \calX_{h}$ and where the function $\tilde V := \tilde V(y;V)\in \calX_{h}$ are given by the relation
\begin{equation}\label{eq:lapl_sinc}
\int_{\mathbb T^2} \tilde V  R + e^{-y} \int_{\mathbb T^2} \nabla \tilde V \cdot \nabla  R = - \int_{\mathbb T^2} V  R, \quad \forall R \in \calX_h.
\end{equation}
As in the previous case, the integration in $y$ is approximated by a sinc quadrature: given $k>0$ and $M \sim 1/k^2$, we define 
\begin{equation}\label{e:apprx_action}
A_{h,k}(V,W):= 2\frac{\sin( \pi s)}{\pi} k \sum_{\ell = -M}^M  e^{sy_\ell}\int_{\mathbb T^2} (V+\tilde V(y_\ell;V))W \approx \int_{{\mathbb T^2}} (-\Delta)^s V \ W.
\end{equation}

An analysis of this approximation strategy in the more complex case of the integral fractional Laplacian is available in \cite{bonito2019numerical}. 
We do not expand on this further but note for later use that because $\int_{\mathbb T} \tilde V(y,V) = - \int_{\mathbb T^2} V$, we deduce that
\begin{equation}\label{e:cons_A}
A_{h,k}(V,1)= 0, \qquad \forall V \in \calX_h.
\end{equation}

\subsubsection{Approximations of the System Velocity}\label{ss:velocity}

We now discuss the approximation of the velocity $\bu$ in \eqref{eq:sqg_frac} for a given approximation $\Theta \in \calX_{h,0}$ of $\theta \in L^2_\#(\mathbb T^2)$.
It is performed in two steps. First, we use the approximation of the inverse fractional Laplacian \eqref{e:appx_inv} to define $\Psi_k \in \calX_{h,0}$ as
$$
\Psi_k := \frac 1 \pi k \sum_{\ell=-M}^{M} e^{(1-s)y_{\ell}}W(y^{\ell}) \approx \psi = (-\Delta)^{-1/2}\theta,
$$
where $W(y^\ell) \in \mathcal X_{h,0}$ solves
$$
e^{y_{\ell}} \int_{\mathbb T^2} W(y_{\ell}) R  + \int_{\mathbb T^2} \nabla W(y_{\ell}) \cdot \nabla R = \int_{\mathbb T^2} \Theta R, \qquad \forall R \in \calX_h.
$$
Then, the velocity approximation $\bU_k := \bU_k(\Theta) \in \left\lbrack\calX_h\right\rbrack^2$ is defined as the componentwise Cl\'ement interpolant \cite{MR0400739}, see also \cite{MR1011446}, of $\nabla^\perp \Psi_k$.
Notice that this construction does not guarantee that $\textrm{div}(\bU_k) = 0$.
This possible lack of conservation property will be accounted for in the design of the algorithm for the temperature potential equation below, see for instance Lemma~\ref{lemma:cons}.

\subsection{Approximation of the Temperature Potential Equation \eqref{eq:sqg_temp}}

The  third order (three stages) Strong Stability Preserving Runge-Kutta (SSP-RK3) {\textcolor{white}.}  method \cite{Shu_Osher1988} is advocated for the approximation of the time evolution in \eqref{eq:sqg_temp}.
We recall that one step of the SSP-RK3 scheme on an homogeneous equation $\frac{d}{dt} v = f(v)$ consist of computing $v^{n+1}$ from  $v^n$ as follows:
\begin{align*}
v^{(1)}& : = v^n+\Delta t_{n+1} f(v^n), \\
v^{(2)} &:= \frac 3 4 v^n + \frac 1 4 (v^{(1)}+\Delta t_{n+1} f(v^{(1)})), \\
v^{n+1} &:= \frac 1 3v^n + \frac 2 3 (v^{(2)}+\Delta t_{n+1} f(v^{(2)})).
\end{align*}
Since SSP-RK3 consists of a linear combination of three forward Euler steps, we restrict the  discussion below to the construction of the latter.
The finite element method for the space discretization is based on the finite element spaces \eqref{e:FEM_space} enhanced with adequate integration formulas and vanishing entropy viscosity stabilizations. 
As we shall see, these choices lead to a method satisfying a maximum principle when $\varkappa=0$ (see Theorems~\ref{th:mp} and \ref{t:FCT}) while retaining in practice the second order accuracy (see Section~\ref{sec:numerical_tests}).

\subsubsection{The Conservative Galerkin Method}

The time interval $[0,T]$ is split onto $N$ intervals of variable length $\Delta t_n$, $n=1,...,N$ and we set $t_n := \sum_{m=1}^n \Delta t_m$, $n=0,...,N$ to denote the breakpoints of this subdivision.
Let $\Theta^0 \in \mathcal X_{h,0}$ be an approximation of the initial temperature potential $\theta_0 \in L^2_\#(\mathbb T^2)$.
We compute $\Theta_k^{n}$, $n=1,...,N$ recursively as detailed now. 
Given the temperature approximation $\Theta_{k}^{n} \in \mathcal X_{h,0}$ and the velocity approximation $\bU_k^{n}:=\bU_k(\Theta_k^{n}) \in  \left\lbrack\calX_h\right\rbrack^2$ (see Section~\ref{ss:velocity}), we define $\Theta_k^{n+1} \in \calX_h$ as the solution to
\begin{equation}\label{eq:sqg:femA}
  \int_{\mathbb T^2} \frac{\Theta_k^{n+1} - \Theta_k^{n}}{\Delta t_{n+1}}  \varphi_i   + \int_{\mathbb T^2} \bU_k^{n} \SCAL \GRAD\Theta_k^n \ \varphi_i
  + \varkappa  A_{h,k}(\Theta_k^n,\varphi_i) = 0, \qquad 1\leq \varphi_i \leq I.
\end{equation}

In general $ \int_{\mathbb T^2}  \Theta_k^{n+1} \not = 0$ due to non-conservative approximation of the velocity, i.e. $\textrm{div}(\bU_k^n) \not = 0$. To circumvent this issue, we follow \cite[Sec.~3.2]{Guermond_Popov_2016a} and replace the flux $\bU^n_k \Theta^n_k$ by its linear interpolation $\sum_{j=1}^I \bu_j^n \theta_j^n \varphi_j$ with $\bu_j^n:= \bU_k^n(\bx_j)$ and $\theta^n_j := \Theta_k^n(\bx_j)$. The velocity term in \eqref{eq:sqg:femA} is thus approximated by
\[
	\int_{{\mathbb T^2}} \bU^n_k \SCAL \GRAD\Theta^n_k \ \varphi_i 
		\approx \sum_{j=1}^I  \theta^n_j\bu_j^n \SCAL \int_{{\mathbb T^2}} \GRAD \varphi_j \varphi_i 
	= \sum_{j=1}^I \bu_j^n \SCAL \bc_{ij} \theta^n_j,  \quad i=1,\ldots,I,
\]
where we introduced the notation  $\bc_{ij}:=\int_{{\mathbb T^2}} \GRAD \varphi_j \varphi_i$. In turn, \eqref{eq:sqg:femA} reduces to a system of equations for the coefficient $(\theta_j^{n+1})_{j=1}^I$ of $\Theta_k^{n+1} \in \calX_h$, namely
\begin{equation}\label{eq:sqg:fem}
 \sum_{j=1}^I m_{ij}\frac{\theta_j^{n+1}-\theta_j^n}{\Delta t_{n+1}} +  \sum_{j=1}^I \bu_j^n \SCAL \bc_{ij} \theta^n_j  + \varkappa  \sum_{j=1}^I \theta_j^n A_{h,k}(\varphi_j,\varphi_i) = 0, \quad i=1,\ldots,I.
 \end{equation}
Notice that the relations 
\begin{equation}\label{e:cij_prop}
\bc_{ij} = - \bc_{ji}\quad \textrm{and} \qquad  \sum_{j=1}^I \bc_{ji} = 0
\end{equation} 
hold. 
As a consequence,
\begin{equation}\label{e:u_cons}
\sum_{i=1}^I  \sum_{j=1}^I \bu_j^n \SCAL \bc_{ij} \theta^n_j = \sum_{i=1}^I  \sum_{j=1}^I \left(\bu_j^n \bc_{ij}\theta^n_j  - \bu_i^n \bc_{ij}\theta^n_i\right)  = 0
\end{equation}
and we now have, recalling the definition \eqref{e:m} of $m_{ij}$, 
$$
\int_{\mathbb T^2} \Theta_k^{n+1} = \sum_{i,j = 1}^{I} m_{ij} \theta_j^{n+1} = \sum_{i,j = 1}^{I} m_{ij} \theta_j^{n} = \int_{\mathbb T^2} \Theta_k^{n} = 0
$$
thanks to \eqref{e:cons_A} as well.
From this, we deduce that $\Theta_k^{n+1} \in \calX_{h,0}$.

The numerical scheme \eqref{eq:sqg:fem} does not satisfy a maximum principle, it is actually not even stable in $L^\infty(\mathbb T^2)$.
In the next sections we modify the above scheme and obtain low order method satisfying a maximum principle (or invariant domain property when $\varkappa>0$).

\subsubsection{Low Order Scheme}\label{s:num_quad}

We follow the approach in \cite{Guermond_Nazarov_2014} modifying  \eqref{eq:sqg:fem} with appropriate mass lumping quadratures and incorporating a low order (graph) viscosity. 
When $\varkappa = 0$, we show that the resulting scheme satisfies a discrete maximum principle property
  \[
\min_{j \in \calI(i)} \theta_j^n \le \theta_i^{n+1}  
  \le \max_{j \in \calI(i)} \theta_j^n.
  \]
When $\varkappa >0$, we cannot guaranteed that $\theta_i^{n+1}$ strictly lies in $[\min_{j \in \calI(i)}\! \theta_j^n,\min_{j \in \calI(i)}\! \theta_j^n]$ without additional assumptions. We postpone this discussion to Remark~\ref{r:viscouscase} below.

We start with a mass lumping quadrature formula (see the definitions \eqref{e:m}) for the time derivative term 
$$ 
\sum_{j=1}^I m_{ij}\frac{\theta_j^{n+1}-\theta_j^n}{\Delta t_{n+1}} \approx m_i \frac{\theta_i^{n+1}-\theta_i^n}{\Delta t_{n+1}}
$$ 
and the diffusion term 
$$ 
\sum_{j=1}^I \theta_j^n  A_{h,k}(\varphi_j,\varphi_i) \approx m_i A_i(\Theta_k^n),
$$
where
\begin{equation}\label{eq:Di}
A_i(\Theta_k^n) := 2\frac{\sin( \pi s)}{\pi} k \sum_{\ell = -M}^M  e^{sy_\ell}(\theta_i^n+\tilde v_i(y_\ell;\Theta_k^n))
\end{equation}
and $\tilde v_i := \tilde v_i(y_\ell;\Theta_k^n)$, $i=1,...,I$, satisfies 
$$
m_i \tilde v_i  + e^{-y_\ell} \sum_{j \in I(i)} \tilde v_j \int_{\mathbb T} \nabla \varphi_j \cdot \nabla \varphi_i = - m_i \theta_i^n;
$$
compare with \eqref{e:apprx_action} and \eqref{eq:lapl_sinc}. Notice that because $A_{h,k}(\Theta_k^n,1)=0$, see \eqref{e:cons_A}, we have
\begin{equation}\label{e:Ai_prop}
\sum_{i=1}^I m_i A_i(\Theta_k^n) = 0,
\end{equation}
instrumental property to preserve the average of the buoyancy after each time step (see Lemma~\ref{lemma:cons}).
It is well known that standard continuous Galerkin methods are not stable in the approximation of first order hyperbolic systems \cite[Chapter~5]{Ern_Guermond_2004}. 
We propose here an artificial (vanishing) graph viscosity approach to stabilize our system, see \cite[Sec.3.2]{Guermond_Popov_2016} and \cite[Sec.4.2]{Guermond_Nazarov_2014} and incorporate a diffusing term
$$
\sum_{j \in \calI(i)} d_{ij}^{L,n} \theta_j^n
$$
in \eqref{eq:sqg:fem}. The coefficients $d_{ij}^{L,n}$ are defined for $i\not = j$ as
\begin{align}\label{eq:dijL}
	d^{L,n}_{ij} 
	:=
	 \max(\lambda_{\max}(\bn_{ij}, \theta_i^n, \theta_j^n) \|\bc_{ij}\|_{\ell^2},  
	 \lambda_{\max}(\bn_{ji}, \theta_j^n, \theta_i^n) \|\bc_{ji}\|_{\ell^2}),
\end{align}
where the local maximum wave speed are given by
\begin{equation}\label{def:max_speed}
 	\lambda_{\max}:=\lambda_{\max}(\bn_{ij}, \theta_i^n, \theta_j^n) 
	:= \max(|\bu_i^n\SCAL \bn_{ij}|, |\bu_j^n\SCAL \bn_{ij}|),
\end{equation}
with $\bn_{ij}:= \bc_{ij} / \|\bc_{ij}\|_{\ell^2}$ and $\| \bc_{ij} \|_{\ell^2}$ denotes the Euclidian norm of the vector $\bc_{ij} \in \mathbb R^2$.

For $i=j$, we set
\begin{align}\label{eq:diiL}
	d^{L,n}_{ii} 
	:= -\!\!\!\! \sum_{i\not=j \in \calI(i)} d_{ij}^{L,n}.
\end{align}
For future reference, we record the properties of the artificial viscosity coefficients:
\begin{equation}\label{eq:dij_prop}
d_{ij}^{L,n} \ge 0, \qquad d_{ij}^{L,n} = d_{ji}^{L,n},   \quad \textrm{and} \quad \sum_{j\in\calI(i)}d_{ij}^{L,n} = \sum_{i\in\calI(j)}d_{ij}^{L,n}  = 0.
\end{equation}

We are now in position to define the low order scheme associated with \eqref{eq:sqg:fem}: for $\Theta^n_{k} = \sum_{i =1}^I \theta_i^n \varphi_i$ and 
$\bU_{k}^n = \sum_{i=1}^I \bu_i^n \varphi_i$, we determine $\Theta^{L,n+1}_{k} = \sum_{i =1}^I \theta_i^{L,n+1} \varphi_i$ from the independent and explicit relations
\begin{equation}\label{eq:sqg:lo}
\begin{split}
m_i\theta_i^{L,n+1}  &= m_i\theta^{n}_i - \Delta t_{n+1} \sum_{j\in \calI(i)} \bu_j^n \SCAL \bc_{ij} \theta^n_j \\
&\quad - \varkappa 
	\Delta t_{n+1} m_iA_i(\Theta_k^n) + \Delta t_{n+1} \sum_{j \in I(i)} d_{ij}^n \theta_j^n, \qquad 1\leq i \leq I.
\end{split}
	\end{equation}

The properties of the above scheme are discussed next.
We start with a lemma ensuring that the discrete scheme preserves the average of the buoyancy, critical property to define the velocity (see Section \ref{ss:velocity}).

\begin{lemma}\label{lemma:cons}
The low order scheme defined by the relations \eqref{eq:sqg:lo} is conservative, i.e. 
$$
\int_{\mathbb T^2} \Theta_k^{L,n} = \int_{\mathbb T^2} \Theta_h^n.
$$
\end{lemma}
\begin{proof}
After summing for $i=1,...,I$ the relation \eqref{eq:sqg:lo} and using the conservation properties \eqref{e:u_cons}, \eqref{eq:dij_prop} and \eqref{e:Ai_prop},  we realize that
\begin{align*}
	 \sum_{i=1}^I m_i\theta^{L,n+1}_i = \sum_{i=1}^I m_i\theta^{n}_i.
\end{align*}
Hence,  
$$
\int_{\mathbb T^2} \Theta_k^{L,n+1} = \sum_{i=1}^I m_i\theta^{L,n+1}_i = \sum_{i=1}^I m_i\theta^{n}_i = \int_{\mathbb T^2} \Theta_k^n,
$$
which is the desired estimate.
\end{proof}

We now turn our attention to the discrete maximum principle when $\varkappa =0$.
The discrete maximum property requires a CFL type condition to hold, namely there exists a real number $0<\textrm{CFL}\leq \frac 1 2$ such that the time step are selected (on the fly) to satisfy
\begin{equation}\label{e:CFL}
\frac{\Delta t_{n+1}}{m_i} \sum_{i\not=j \in \calI(i)} d_{ij}^{L,n} \le  \textrm{CFL}.
\end{equation}
Note that $m_i \sim h^2$, $d_{ij}^{L,n} \sim \lambda^n h$, where $\lambda^n$ is a characteristic (local velocity), and thus the above condition requires that for the computation of $\Theta_k^{n+1}$, the time step $\Delta t_{n+1}$ is selected so that $\Delta t_{n+1} \lambda^n/h$ is sufficiently small. 

\begin{theorem}[Discrete maximum principle]\label{th:mp}
  Let us assume that $\varkappa = 0$ and assume that condition \eqref{e:CFL} holds for some $0<\textrm{CFL}\leq \frac 1 2$.
  Then the solution of the low order scheme \eqref{eq:sqg:lo} satisfies
 \begin{equation}\label{e:max_principle}
 \min_{j \in \calI(i)} \theta_j^n \le \theta_i^{L,n+1}  
  \le \max_{j \in \calI(i)} \theta_j^n
  \end{equation}
 for all $i=1,\ldots, I$.
  
\end{theorem}

\begin{proof}
Using the conservation properties \eqref{e:u_cons} and \eqref{eq:dij_prop}, we rewrite \eqref{eq:sqg:lo} as
\begin{align*}
	\theta^{L,n+1}_i =  \theta^{n}_i - \frac{\Delta t_{n+1}}{m_i} \sum_{i \not = j\in \calI(i)} (\bu_j^n \theta^n_j - \bu_i^n \theta^n_i) \SCAL \bc_{ij} +  \frac{\Delta t_{n+1}}{m_i} \sum_{i \not = j\in \calI(i)} d^{L,n}_{ij}(\theta^n_j - \theta^n_i)
\end{align*}
or, rearranging the terms, as 
$$
	\theta^{L,n+1}_i = 
\theta^{n}_i\left(1 - \frac{\Delta t_{n+1}}{m_i}\!\! \sum_{i \not=j\in \calI(i)}\!( - \bu_i^n\SCAL \bc_{ij} +  d^{L,n}_{ij}) \right)+ \frac{\Delta t_{n+1}}{m_i} \!\! \sum_{i\not=j\in \calI(i)}\! ( - \bu_j^n\SCAL \bc_{ij} + d^{L,n}_{ij}) \theta^n_j.
$$

We obtain \eqref{e:max_principle} by showing that the right hand side of the above relation is a convex combinations of $\theta^{n}_j$, $j \in \calI(i)$.
To see this, we first note that the coefficients add-up to 1.
 Moreover, from the definition \eqref{eq:dijL} of the low order viscosity coefficients $d_{ij}^{L,n}$, we have $\bu_j^n \cdot \bc_{ij} \leq d_{ij}^{L,n}$ and $-\bu_i^n \cdot \bc_{ij} \leq d_{ij}^{L,n}$.
 The former guarantees that the coefficients in front of the $\theta_j^n$ are positive.
 The latter, in conjunction with the assumption $\textrm{CFL} \leq \frac 1 2$, yields 
\[
	1 - \frac{\Delta t_{n+1}}{m_i} \sum_{i \not=j\in \calI(i)}( - \bu_i^n\SCAL \bc_{ij} +  d^{L,n}_{ij}) 
	\ge
	1 - \frac{\Delta t_{n+1}}{m_i} \sum_{i \not=j\in \calI(i)}2d^{L,n}_{ij}
	\ge 0
\]
and so the coefficient in front of $\theta_i^n$ is positive as well. 
This ends the proof.
\end{proof}

We conclude this section with a remark concerning the viscous case.
\begin{remark}\label{r:viscouscase} 
We have already mentioned that a maximum principle like \eqref{e:max_principle} does not necessarily hold without additional assumption.
However, proceeding as in the proof of Theorem~\ref{th:mp}, we can rewrite the low order scheme as
\begin{align*}
\Theta_i^{\text{L}, n+1} 
	=
	\frac12 
	&\Big[
	\big(
	1 - \frac{2\Delta t_{n+1}}{m_i} \sum_{i \not=j\in \calI(i)}( - \bu_i^n\SCAL \bc_{ij} +  d^{\text{L}, n}_{ij}) 
	\big) \theta_i^n \\
	&+
	\frac{2\Delta t_{n+1}}{m_i} \sum_{i \not=j\in \calI(i)}( - \bu_i^n\SCAL \bc_{ij} +  d^{\text{L}, n}_{ij}) \theta_j^n
	\Big]
	+
	\frac12 
	\Big[
	\theta_i^n - 2\Delta t_{n+1} \sum_{i=1}^Im_i A_i(\Theta_k^n)
	\Big].
\end{align*}
From this we see that if for some real numbers $a<b$ we have $\theta_i^n \in [a,b]$ and $\theta_i^n - 2\Delta t_{n+1} \sum_{i=1}^Im_i A_i(\Theta_k^n) \in [a,b]$, both for $i=1,...,I$, 
then $\theta_i^{L,n+1} \in [a,b]$ for $i=1,...,I$ whenever 
$$
\Delta t_{n+1} \le m_i / (4 \sum_{i\not=j \in \calI(i)} d_{ij}^{\text{L},n}).
$$
The above condition holds provided $0<\textrm{CFL}\leq \frac 1 4$.
This property is called invariant domain in \cite{Guermond_Popov_2016}.

Alternatively, if the triangulation satisfies the acute angle condition $\int_{\mathbb T^2} \nabla \varphi_i \cdot \nabla \varphi_j <0$ for $i\not = j$ along with a restriction on the sinc quadrature, then the viscous low order scheme satisfies the maximum principle property \eqref{e:max_principle}.
This finer analysis is out of the scope of this paper and we refer to \cite{BGSometimes} for additional details.
\end{remark}

\subsubsection{Higher Order Scheme} 
\label{section:EV}
The construction of the higher order scheme starts again from the Galerkin scheme \eqref{eq:sqg:fem} but with an artificial viscosity $d_{ij}^{H,n}$ chosen to vanish at a higher rate (with respect to the meshsize $h$) than \eqref{eq:dijL} and \eqref{eq:diiL} used for the low order scheme.   

Besides being of higher order, the only restriction needed on the artificial viscosity coefficients $d_{ij}^{H,n}$ is that they satisfy the properties \eqref{eq:dij_prop}.
In this work, we propose an  artificial viscosity proportional to the residual of one entropy of the buoyancy equation.
This is commonly referred to as an entropy viscosity method. 
Such strategy for conservation laws was originally proposed by \cite{Guermond_pasquetti_popov_JCP_2011} and was later extended to solve compressible flows \cite{Nazarov_Hoffman_2012, Nazarov_Larcher_2017}. 
A priori error and stability analysis of the method for some entropy functionals were investigated in \cite{Nazarov_2013} and \cite{Bonito_etal_2014}.

Given $\Theta_k^n= \sum_{j=1}^I \theta_j^n \varphi_i$ and $\bU_k^n=\sum_{j=1}^I \bu_j^n \varphi_i$, the higher order scheme consists of finding  $\Theta_k^{H,n+1} = \sum_{j=1}^I \theta_j^{H,n+1} \varphi_i$ from the system of equations
\begin{equation}\label{eq:sqg:ho}
\begin{split}
 \sum_{j=1}^I m_{ij} \theta_j^{H,n+1} = 
  &\sum_{j=1}^I m_{ij}\theta_j^n -  \Delta t_{n+1}\sum_{j=1}^I \bu_j^n \SCAL \bc_{ij} \theta^n_j  \\
  &- \varkappa   \Delta t_{n+1} m_i A_i(\Theta_n) +  \Delta t_{n+1}\sum_{j=1}^I d_{ij}^{H,n} \theta_j^n,
\end{split}
\end{equation}
for $i=1,\ldots,I$ and where the higher order entropy residual viscosity coefficients $d_{ij}^{H,n}$ are yet to be determined.
This is the focus of the remaining part of this section but before embarking in this discussion, we point out that unlike for the lower order scheme \eqref{eq:sqg:lo}, we use the consistent mass matrix for the time derivative term (no mass lumping) to reduce the dispersion error generated by the mass lumping in the low order scheme.

Entropy residuals have been discussed in details in the literature, see for instance \cite{Guermond_pasquetti_popov_JCP_2011}.
In our particular context, for a given $q$ sufficiently smooth, we define the entropy residual by $\calR^n(Q):= \sum_{i=1}^I \calR_i^n(q) \varphi_i$ where
$$
	\calR_i^n(q) 
	:= 
	\int_{{\mathbb T^2}}
	\Big(
	\frac{q - \Theta^{n}_k}{\Delta t_{n+1}}
	+ \bU^n_k\SCAL\GRAD \Theta^n_k + \varkappa \sum_{j=1}^I  A_j(\Theta^n_k) \varphi_j \Big) \eta'(\Theta_k^n) \varphi_i,
$$
where the entropy function $\eta$ is taken to be $\eta(x):= \frac 1 2 x^2$. 
Note that the action of the operator $(-\Delta)^s$ is not well defined on  $\Theta^n_k$ and is therefore replaced in $\calR_i^n$ by $\sum_{j=1}^I  A_j(\Theta^n_k) \varphi_j$.

One of the difficulty when using residual based viscosity on dynamical systems is the proper handling of the time derivative and in particular what function $q$ to use.
In order to avoid interferences from the time discretization in the computation of the residual, we resort to a novel idea from \cite{Guermond_et_al_2018}, see also \cite{Lu_Nazarov_Fischer_2019}.
To motivate the final expression of the residual we  (formally) consider the solution $\theta^G$ of the following  implicit time discretization
\begin{equation}\label{eq:sqg:galerkin}
\frac{\theta^G-\theta_j^n}{\Delta t_{n+1}} + \bU_k^n \SCAL \GRAD \theta^G  + \varkappa(-\Delta)^{s} \theta^G = 0.
\end{equation}
The entropy residual evaluated at $q=\theta^G$ reads
$$
\calR^n_i (\theta^G)
	= 
	\int_{{\mathbb T^2}}
	\Big(
	-\bU^n_k \SCAL \GRAD\theta^{\text{G}}
  	- \varkappa (-\Delta)^s \theta^{\text{G}} 
	+ \bU^n_k\SCAL\GRAD \Theta^n_k + \varkappa \sum_{j=1}^I  A_j(\Theta^n_k) \varphi_j\Big) \eta'(\Theta_k^n) \varphi_i.
$$
The above expression is not practical because of the cost in computing $\theta^G$. Instead, one can use the plain Galerkin soluton
$\Theta^{G,n+1}_k := \sum_{j=1}^I \theta_j^{G,n+1} \in \calX_h$ defined as the higher order scheme but without artificial viscosity
\begin{equation}\label{eq:sqg:G}
 \sum_{j=1}^I m_{ij}\frac{\theta_j^{G,n+1}-\theta_j^n}{\Delta t_{n+1}} = -  \sum_{j=1}^I \bu_j^n \SCAL \bc_{ij} \theta^n_j  - \varkappa  m_i A_i(\Theta_k^n), \quad i=1,\ldots,I,
\end{equation}
which leads to the final expression for the entropy residual
\begin{equation}\label{eq:res}
	\calR^n_i 
	:= 
	\int_{{\mathbb T^2}}
	\Big(
	\bU^n_k \SCAL \GRAD (\Theta_k^n-\Theta_k^{\text{G},n+1})
  	+ \varkappa \sum_{j=1}^I  (A_j(\Theta_k^n)-A_j(\Theta_k^{G,n+1})) \varphi_j
	\Big)\eta'( \Theta_k^n) \varphi_i.
\end{equation}
Then, the high order nonlinear viscosity in \eqref{eq:sqg:ho} is defined by 
\begin{align}\label{eq:dijH}
	d^{\text{H}, n}_{ij} 
	:=
	\min
	\Big(
	d^{\text{L}, n}_{ij}, 
	c_{\text{EV}}\max\Big(
	\frac{\calR^n_i}{\widetilde{\eta^n_i}}, 
	\frac{\calR^n_j}{\widetilde{\eta^n_j}}
	\Big)
	\Big),	
\end{align}
where $c_{\text{EV}}$ is the stablization parameter (typically $0.1 \leq c_{\text{EV}} \leq 1$.
The normalization coefficient $\widetilde{\eta^n_i}$ in \eqref{eq:dijH} are given by
\begin{equation}\label{eq:normalization}
  \widetilde{\eta^n_i} :=
  \max(
  \big|
  \max_{j\in\calI(i)}\eta(\Theta^n_j) - \min_{j\in\calI(i)}\eta(\Theta^n_j)
  \big|, \epsilon |\eta(\Theta^n_i)|),
\end{equation}
with $\epsilon := 10^{-8}$ (or below the scheme accuracy) is a small safety factor.
We refer to Section~\ref{sec:numerical_tests} for a discussion on the effect of $c_{\textrm{EV}}$ and on the normalization.

\subsubsection{Flux corrected transport (FCT) limiting}
\label{section:FCT}

In the above sections, we introduced two methods: a first-order maximum principle (or invariant domain) preserving scheme and a high order nonlinear viscosity scheme. The FCT algorithm below, first introduced by \cite{Boris_books_JCP_1973}, ensures that the high order solution satisfies the discrete maximum principle (or invariant domain property). 

We relate the high and low order schemes by subtracting \eqref{eq:sqg:lo} from \eqref{eq:sqg:ho}:
$$
 \sum_{j \in \calI(i)} m_{ij} \theta_j^{H,n+1} = m_i \theta_i^{\textrm{L},n+1} + \sum_{j \in \calI(i)}m_{ij} (\theta_j^n-\theta_i^n)
 + \Delta t_{n+1}\sum_{j \in \calI} (d_{ij}^{H,n}-d_{ij}^{L,n}) \theta_j^n.
$$
Note that to derive the above relation, we used the definition  $m_i = \sum_{j\in \calI(i)} m_{ij}$. 
Adding $m_{i} \theta_i^{H,n+1}$ on both sides of the equation, we get
\begin{equation}\label{e:H-L}
\begin{split}
m_{i} \theta_i^{H,n+1} = & m_i \theta_i^{\textrm{L},n+1} + \sum_{j \in \calI(i)}m_{ij} (\theta_j^n-\theta_i^n)- \sum_{j \in \calI(i)}m_{ij} ( \theta_j^{\textrm{H},n+1} - \theta_i^{\textrm{H},n+1})\\
& + \Delta t_{n+1}\sum_{j\in \calI(i)}^I (d_{ij}^{H,n}-d_{ij}^{L,n}) \theta_j^n\\
& =: m_i \theta_i^{\textrm{L},n+1} + \Delta t_{n+1} \sum_{j \in \calI(i)} \calA_{ij}.
\end{split}
\end{equation}
The coefficient $\calA_{ij}$ can be rewritten  using the conservative properties   \eqref{eq:dij_prop}, \eqref{eq:dijH} of the artificial diffusions coefficient as
\begin{equation}\label{eq:aij}
\begin{aligned}
\calA_{ij} = & - \frac{m_{ij}}{\Delta t_{n+1}}
	\Big(
	(\theta^{\text{H},n+1}_j - \theta^{n}_j)
	-
	(\theta^{\text{H},n+1}_i - \theta^{n}_i)
	\Big) \\
	&+ (d_{ij}^{\text{H},n} - d_{ij}^{\text{L},n}) (\theta^{n}_j - \theta^{n}_i).
\end{aligned}
\end{equation}
From this representation, one sees that $\calA_{ij}=-\calA_{ji}$.

The low order solution as proven earlier preserves the discrete maximum principle ($\varkappa=0$).
$$
\theta_\textrm{min}^n := \min_{j=1,...,I} \theta_j^n \leq \theta^{\text{L},n+1}_i \leq  \max_{j=1,...,I} \theta_j^n  =:\theta_\textrm{max}^n.
$$
However, the high order solution may violate this maximum principle. The idea of Zalesak \cite{Zalesak_1979} is to introduce a limiter matrix  of coefficient $\calL_{ij} \geq 0$  to guarantee that the high order solution remains within $[\theta_\textrm{min}^n,\theta_\textrm{max}^n]$ while  retaining high-order accuracy. 
To make this more precise, we write
$$
\theta_{\min}^n 
 = 
\theta_i^{\text{L}, n+1} + (\theta_{\min}^n  - \theta_i^{\text{L}, n+1}) 
= 
\theta_i^{\text{L}, n+1} + \frac{\Delta t_{n+1}}{m_i} Q_i^-,
$$
where $Q_i^- := \frac{m_i}{\Delta t_{n+1}} (\theta^{\min,n}_i - \theta^{\text{L},n+1}_i)$, $i=1,...,I$.
Furthermore, using the notations $P_i^- := \sum_{j\in \calI(i)} \min\{0, \calA_{ij}\}$ and $R_i^-:= \min \Big\{ 1, \frac{Q_i^-}{P_i^-}\Big\}$ for $i=1,...,I$, we deduce that
\begin{equation}\label{eq:theta_min}
\theta_{\min}^n  \leq \theta_i^{\text{L}, n+1} + \frac{\Delta t_{n+1}}{m_i}
\sum_{j\in \calI(i), \ \calA_{ij} \le 0}  R_i^-  \calA_{ij}.
\end{equation}
Similarly, upon defining $Q_i^+ := \frac{m_i}{\Delta t_{n+1}} (\theta_{\max}^n - \theta^{\text{L},n+1}_i)$, $P_i^+ := \sum_{j\in \calI(i)} \max\{0, \calA_{ij}\}$ and $R_i^+ := \min \Big\{ 1, \frac{Q_i^+}{P_i^+}\Big\}$, we have
\begin{equation}\label{eq:theta_max}
\theta_{\max}^n  \geq  \theta_i^{\text{L}, n+1} + \frac{\Delta t_{n+1}}{m_i}
\sum_{\substack{j\in \calI(i) \\ \calA_{ij} \ge 0}}  R_i^+  \calA_{ij}.
\end{equation}
In view of \eqref{eq:LHlij}, \eqref{eq:theta_min} and \eqref{eq:theta_min}, we define 
\begin{equation}\label{eq:lij}
		\calL_{ij} := 
		\begin{cases}
		\min\{R_i^+, R_j^-\}, \quad \text{if } \calA_{ij} \ge 0, \\
		\min\{R_i^-, R_j^+\}, \quad \text{otherwise}
		\end{cases}
	\end{equation}
and the coefficients of the FCT solution $\Theta_k^{n+1}$ are obtained from the relation
\begin{equation}\label{eq:LHlij}
	 \theta^{n+1}_i 
	= \theta^{\text{L},n+1}_i + \frac{ \Delta t_{n+1}}{m_i} \sum_{j\in \calI(i)} \calL_{ij}\calA_{ij};
\end{equation}
compare to \eqref{e:H-L}. 
We have the following result.

\begin{theorem}\label{t:FCT}
The FCT solution $\Theta_k^{n+1}= \sum_{i=1}^I \theta_i^{n+1} \varphi_i$ of \eqref{eq:LHlij} with $\Theta_k^{\textrm{L},n+1} = \sum_{i=1}^I \theta_i^{\textrm{L},n+1} \varphi_i$ satisfies
\begin{equation}\label{e:FCTc}
\int_{\mathbb T^2} \Theta_k^{n+1} = \int_{\mathbb T^2} \Theta_k^{L,n+1}.
\end{equation}
Furthermore, if for some $a,b \in \mathbb R$ the lower order scheme satisfies $\theta_i^{\textrm{L},n+1} \in [a,b]$ for all $i=1,...,I$ and all $n\geq 1$, then 
\begin{equation}\label{e:FCTr}
\theta_i^{n+1} \in [a,b], \quad \forall i=1,...,I, \ \forall n\geq 1.
\end{equation}
\end{theorem}

\begin{proof}
For the conservative property, observe that $\calA_{ij}$ is skew-symmetric and $\calL_{ij}$ is symmetric. Then multiplying \eqref{eq:LHlij} by $m_i \varphi_i$ and summing over $i$ we get \eqref{e:FCTc}.
Relation \eqref{e:FCTr} follows directly from the definition of $\calL_{ij}$. 
The proof is complete.
\end{proof}

We conclude this section with a summary of one Euler step for the approximation of the buoyancy equation.

\begin{algorithm}[H]
\renewcommand{\algorithmicrequire}{\textbf{Input:}}
\renewcommand{\algorithmicensure}{\textbf{Output:}}
\caption{Limiting algorithm for potential temperature equation}
\label{fct}
\begin{algorithmic}[1]
\Require $\Theta_k^{n}$, $\bU_k^n$, $\varkappa$ and $\Delta t_{n+1}$
\Ensure $\Theta_k^{n+1}$
\State Compute  $d^{\text{L},n}_{ij}$ from \eqref{eq:dijL} and the low order solution $\theta^{\text{L},n+1}_i$ defined by \eqref{eq:sqg:lo};
\State Compute  $\Theta^{\textrm{G},n+1}_k$ defined by \eqref{eq:sqg:G} to construct $d^{\text{H},n}_{ij}$  in \eqref{eq:dijH};
\State Compute  the higher order solution $\theta^{\text{H},n+1}_i$ defined by \eqref{eq:sqg:ho};
\State Compute the matrix $\calA_{ij}$ in \eqref{eq:aij} and $\calL_{ij}$ in \eqref{eq:lij};
\State Compute $\Theta_k^{n+1}$ using \eqref{eq:LHlij}.
\end{algorithmic}
\end{algorithm}

\section{Numerical Illustrations}
\label{sec:numerical_tests}

In this section, we solve several benchmark problems to validate the proposed numerical scheme and present novel insightful simulations. 
The smooth convection problem in Section~\ref{sec:smooth_conv} validates the entropy residual viscosity model \eqref{eq:dijH} and observe that the FCT scheme preserve the high order accuracy of the high order scheme.
The discretization of the fractional diffusion operator is investigated in Section~\ref{sec:smooth_frac}.
In Sections~\ref{s:vortex} and \ref{s:viscous_transition}, we perform standard benchmarks for the SQG system: rotating vortices and initial data with saddle structures leading to sharp transitions.
We conclude with a turbulence study in Section~\ref{sec:2d_turbulence} and show that our numerical scheme exhibit the theoretical predictions of the Kolomogorov energy decay rate.

In order to plot the evolution of the buoyancy approximation, we denote by $\Theta_k(t)$ the continuous piecewise time reconstruction defined  by $\Theta_k(t)|_{[t_n,t_{n+1}]}:=\Theta_k^n + (\Theta_k^{n+1}-\Theta_k^n) (t-t_n)/\Delta t_{n+1}$.

\subsection{Smooth convection problem}\label{sec:smooth_conv}
We consider the inviscid SQG equations, \ie $\varkappa =0$, the convection field is given by $\bu= (1,1)^\top$, and initial data is a smooth function defined as 
\[
\theta_0(x_1,x_2) = \sin x_1 \sin x_2 + \cos x_2.
\]
This is a pure convection problem with constant transport illustrating the differences between the low order, the higher order and FCT schemes.

We run the problem on a sequence of meshes until the final time $T=2\pi$. The time steps is uniform over the entire simulation and chosen so that $\textrm{CFL}= 0.2$, see \eqref{e:CFL}.
Note that in view of Theorem~\ref{t:FCT}, a larger $\textrm{CFL}$ value than needed for the stabilized schemes is chosen to guarantee the stability of the Galerkin scheme used for comparison purposes. 
The nonlinear entropy residual parameter in \eqref{eq:dijH} is set to $c_{\textrm{EV}}=1$.

The results of the numerical simulation are collected in Table~\ref{table:2Dconvection}.
The first-row block corresponds to Galerkin solution \eqref{eq:sqg:G}, \ie without any stabilization terms. The second and third-row blocks correspond to the entropy viscosity solution described in Section~\ref{section:EV} and the FCT solution described in Section~\ref{section:FCT}.  We compute the errors at the final time for the $L^1(\mathbb T^2)$-, $L^2(\mathbb T^2)$- and $L^\infty(\mathbb T^2)$-norms for several spacial resolutions (uniform triangulations) along with their associated rate of convergence.
We observe a second-order convergence rate in the $L^1(\mathbb T^2)$- and $L^2(\mathbb T^2)$-norms but the entropy viscosity solution and the FCT solution deliver suboptimal rate in the $L^{\infty}(\mathbb T^2)$-norm.
Obtaining optimal rates in the $L^{\infty}(\mathbb T^2)$-norm for limited solutions is notoriously difficult, see \eg \cite{Guermond_Popov_2017}.
However, we report in the fourth and fifth row blocks of Table~\ref{table:2Dconvection}, FCT simulations obtained using a different normalization term
\begin{equation}\label{eq:normalization2}
  \widetilde{\eta_2^n} :=
  \max( \widetilde{\eta^n_i},  |\eta(\Theta^n_i)|)
\end{equation}
and leading to optimal second order convergence rates in all norms.
Although not optimal in the maximum norm, we use for the rest of the paper the normalization \eqref{eq:normalization} because of its robustness on nonlinear problems.

The values of the kinetic energy and helicity \eqref{eq:kinetic_helicity} are given in the last two columns of Table~\ref{table:2Dconvection}. We see that for the finer meshes the method recovers the kinetic energy and helicity to their reference values of $14.8041$ and  
  $26.4241$ computed with the initial condition on the finest mesh. 
  Note that the exact value of the kinetic energy is $3\pi^2/2 \approx 14.8044$.

\begin{table}[htb]\footnotesize \centering 
  \begin{tabular}{|c|c|c|c|c|c|c|c||c|c|}
    \hline
    & \# dofs  & $L^1$ & rate    & $L^2$     & rate   & $L^{\infty}$ & rate& $\calK(\theta)$ & $\calH(\theta)$\\
    \hline
    \multirow{5}{*}{\rotatebox{90}{Galerkin \qquad}}
&      100 &   1.58E+00 &   --   &   3.13E-01 &   --   &   1.44E-01 &   --   & 13.5743 & 23.8296 \\
&      400 &   3.77E-01 &   2.06 &   7.52E-02 &   2.06 &   3.35E-02 &   2.10 & 14.4839 & 25.7446 \\
&     1600 &   9.39E-02 &   2.01 &   1.87E-02 &   2.01 &   8.23E-03 &   2.02 & 14.7235 & 26.2526 \\
&     6400 &   2.35E-02 &   2.00 &   4.68E-03 &   2.00 &   2.06E-03 &   2.00 & 14.7841 & 26.3816 \\
&    25600 &   5.87E-03 &   2.00 &   1.17E-03 &   2.00 &   5.14E-04 &   2.00 & 14.7993 & 26.4139 \\
&   102400 &   1.47E-03 &   2.00 &   2.93E-04 &   2.00 &   1.29E-04 &   2.00 & 14.8031 & 26.4220 \\
&   409600 &   3.67E-04 &   2.00 &   7.31E-05 &   2.00 &   3.21E-05 &   2.00 & 14.8041 & 26.4241 \\
    \hline
    \multirow{5}{*}{\rotatebox{90}{EV \qquad}}
&      100 &   9.26E+00 &   --   &   1.78E+00 &   --   &   6.04E-01 &   --   & 7.43147 & 17.7222 \\
&      400 &   2.46E+00 &   1.91 &   5.36E-01 &   1.73 &   2.39E-01 &   1.34 & 12.7493 & 24.1963 \\
&     1600 &   6.89E-01 &   1.84 &   1.64E-01 &   1.71 &   9.89E-02 &   1.27 & 14.3682 & 25.9460 \\
&     6400 &   1.77E-01 &   1.96 &   4.60E-02 &   1.83 &   3.94E-02 &   1.33 & 14.7046 & 26.3142 \\
&    25600 &   4.52E-02 &   1.97 &   1.26E-02 &   1.87 &   1.59E-02 &   1.31 & 14.7803 & 26.3980 \\
&   102400 &   1.16E-02 &   1.96 &   3.40E-03 &   1.89 &   6.34E-03 &   1.32 & 14.7985 & 26.4181 \\
&   409600 &   2.97E-03 &   1.96 &   9.10E-04 &   1.90 &   2.53E-03 &   1.32 & 14.8029 & 26.4231 \\
    \hline
    \hline
    \multirow{5}{*}{\rotatebox{90}{FCT+EV \quad}}
&      100 &   9.31E+00 &   --   &   1.79E+00 &   --   &   6.07E-01 &   --   & 7.41079 & 17.6890 \\ 
&      400 &   2.50E+00 &   1.90 &   5.40E-01 &   1.73 &   2.42E-01 &   1.33 & 12.7460 & 24.1918 \\
&     1600 &   6.98E-01 &   1.84 &   1.65E-01 &   1.72 &   9.94E-02 &   1.28 & 14.3681 & 25.9458 \\
&     6400 &   1.80E-01 &   1.96 &   4.63E-02 &   1.83 &   3.95E-02 &   1.33 & 14.7046 & 26.3142 \\
&    25600 &   4.59E-02 &   1.97 &   1.27E-02 &   1.87 &   1.58E-02 &   1.32 & 14.7803 & 26.3980 \\
&   102400 &   1.18E-02 &   1.96 &   3.42E-03 &   1.89 &   6.37E-03 &   1.31 & 14.7985 & 26.4181 \\
&   409600 &   3.10E-03 &   1.92 &   9.30E-04 &   1.88 &   2.55E-03 &   1.32 & 14.8029 & 26.4231 \\
    \hline
    \hline
    \multirow{5}{*}{\rotatebox{90}{EV+ $\widetilde{\eta_2^n}$\quad }}
&      100 &   8.68E+00 &   --   &   1.70E+00 &   --   &   5.61E-01 &   --   & 7.94752 & 14.2740 \\
&      400 &   1.73E+00 &   2.33 &   3.75E-01 &   2.18 &   1.54E-01 &   1.86 & 13.4028 & 23.8887 \\
&     1600 &   3.26E-01 &   2.41 &   6.88E-02 &   2.45 &   3.10E-02 &   2.31 & 14.6079 & 26.0524 \\
&     6400 &   4.83E-02 &   2.75 &   1.03E-02 &   2.74 &   5.15E-03 &   2.59 & 14.7708 & 26.3583 \\
&    25600 &   8.05E-03 &   2.59 &   1.73E-03 &   2.57 &   8.82E-04 &   2.54 & 14.7977 & 26.4111 \\
&   102400 &   1.64E-03 &   2.30 &   3.47E-04 &   2.32 &   1.68E-04 &   2.39 & 14.8029 & 26.4217 \\
&   409600 &   3.82E-04 &   2.10 &   7.86E-05 &   2.14 &   3.62E-05 &   2.21 & 14.8043 & 26.4246 \\
\hline
    \hline
    \multirow{5}{*}{\rotatebox{90}{FCT+EV+$\widetilde{\eta_2^n} $ \quad}}
&      100 &   8.81E+00 &   --   &   1.73E+00 &   --   &   5.67E-01 &   --   & 7.96202 & 14.3007 \\
&      400 &   1.92E+00 &   2.20 &   3.95E-01 &   2.13 &   1.52E-01 &   1.90 & 13.4111 & 23.9015 \\
&     1600 &   3.76E-01 &   2.35 &   7.70E-02 &   2.36 &   3.13E-02 &   2.28 & 14.6085 & 26.0534 \\
&     6400 &   6.94E-02 &   2.44 &   1.34E-02 &   2.52 &   5.42E-03 &   2.53 & 14.7709 & 26.3585 \\
&    25600 &   1.51E-02 &   2.20 &   2.78E-03 &   2.27 &   1.03E-03 &   2.39 & 14.7977 & 26.4111 \\
&   102400 &   3.52E-03 &   2.10 &   6.47E-04 &   2.10 &   2.33E-04 &   2.15 & 14.8029 & 26.4217 \\
&   409600 &   1.05E-03 &   1.75 &   2.06E-04 &   1.65 &   1.96E-04 &   0.25 & 14.8043 & 26.4246 \\
    \hline
 \end{tabular}
   \caption{Convergence tests on smooth convection problem with $\text{CFL}=0.2$. Comparison between the Galerkin, Entropy Viscosity (EV), EV with FCT, EV with normalization \eqref{eq:normalization2} and EV with FCT and normalization \eqref{eq:normalization2} schemes. Errors at time $t=2\pi$ in the $L^1(\mathbb T^2)$, $L^2(\mathbb T^2)$, and $L^\infty(\mathbb T^2)$ norms are reported for different spacial resolutions along with the corresponding convergence rates. 
   The last two columns along reports the Kinetic energy and Helicity. The reference values for the kinetic energy and helicity are $\calK(\theta)=14.8041$ and  
  $\calH(\theta) = 26.4241$.}
  \label{table:2Dconvection}
\end{table}

\subsection{Smooth fractional diffusion problem}\label{sec:smooth_frac}

In this section we approximate a purely fractional diffusion problem
$$
\p_t \theta + \frac{1}{1000} (-\Delta)^{\frac 1 4} \theta = 0, \qquad \textrm{in }  {\mathbb T^2}\times (0,\pi),
$$
supplemented with the initial condition $\theta(x_1,x_2, 0) = e^{-2^{\frac 1 4} \varkappa t} \sin x_2 \cos x_1$.
The latter is a scaled eigenfunction of the Laplacian. 
Hence, in view of the definition \eqref{ss:FracLap}, the exact solution $\theta$ is given by
\[
	\theta(x_1,x_2,t) = e^{-\frac t {1000} 2^{\frac 1 4}} \sin x_2 \cos x_1.
\]

In Table~\ref{table:diffusion}, we report the errors of the Galerkin method \eqref{eq:sqg:G} (with $\bU_k^n \equiv 0$) evaluated at time $t=\pi$ in the $L^1(\mathbb T^2)$-, $L^2(\mathbb T^2)$- and $L^\infty(\mathbb T^2)$-norms on a sequence of uniformly refined subdivisions. Two different choice of the sinc quadrature parameters are investigated, see \eqref{e:apprx_action}.
The time step is set to be $\Delta t = 0.1 h$
Second order convergence rates are observed in all norm when the sinc quadrature parameters are chosen to be $k=0.8$, $M=12$. However, the rate of convergence in the $L^\infty(\mathbb T^2)$ norm is reduced for the finer set of sinc quadrature parameters. Note that the convergence in the $L^\infty$-norm is not analyzed in \cite{bonito2019numerical}.
From now on, the sinc quadrature parameters are set to $k=0.8$ and $M=12$.

\begin{table}[htb]\centering 
  \begin{tabular}{|c|c|c|c|c|c|c|c|}
    \hline
&    \# dofs  & $L^1$ & rate    & $L^2$     & rate   & $L^{\infty}$ & rate\\
    \hline
    \multirow{5}{*}{\rotatebox{90}{$k\!=\!0.2$, $M\!=\!62 \,\,$}}
&      100 &   1.28E+00 &   --   &   2.51E-01 &   --   &   8.44E-02 &   --   \\ 
&      400 &   3.25E-01 &   1.98 &   6.36E-02 &   1.98 &   2.14E-02 &   1.98 \\
&     1600 &   8.19E-02 &   1.99 &   1.59E-02 &   2.00 &   6.08E-03 &   1.82 \\
&     6400 &   2.07E-02 &   1.99 &   4.04E-03 &   1.98 &   2.01E-03 &   1.59 \\
&    25600 &   5.41E-03 &   1.94 &   1.08E-03 &   1.90 &   9.13E-04 &   1.14 \\
&   102400 &   1.50E-03 &   1.85 &   3.36E-04 &   1.69 &   6.97E-04 &   0.39 \\
    \hline
    \hline
    \multirow{5}{*}{\rotatebox{90}{$k\!=\! 0.8$, $M\!=\!12\,\,$}}
&      100 &   1.26E+00 &    --  &   2.47E-01 &    --  &   8.28E-02 &    --  \\
&      400 &   3.20E-01 &   1.97 &   6.27E-02 &   1.98 &   2.10E-02 &   1.98 \\
&     1600 &   7.97E-02 &   2.01 &   1.55E-02 &   2.01 &   5.43E-03 &   1.95 \\
&     6400 &   1.88E-02 &   2.08 &   3.67E-03 &   2.08 &   1.36E-03 &   2.00 \\
&    25600 &   3.55E-03 &   2.40 &   7.27E-04 &   2.34 &   3.40E-04 &   2.00 \\
&   102400 &   1.01E-03 &   1.81 &   2.00E-04 &   1.86 &   1.21E-04 &   1.49 \\
     \hline
 \end{tabular}
   \caption{Effect of the sinc quadrature parameters on a smooth fractional diffusion problem. 
   When the sinc quadrature parameters are chosen too fine, the rate of convergence of the Galerkin finite element approximation deteriorates in $L^\infty(\mathbb T^2)$ but not in $L^1(\mathbb T^2)$ nor in $L^2(\mathbb T^2)$.}
    \label{table:diffusion}
\end{table}

\subsection{Vortex rotation}\label{s:vortex}

When the geometry of the level set of the buoyancy is simple and does not contain an hyperbolic saddle, then the solution to the SQG system \eqref{eq:sqg_temp}, \eqref{eq:sqg_frac} does not exhibit singularities \cite{constantin1994SQG} even when $\varkappa = 0$ as chosen in this section.
To illustrate this, we follow  \cite{Karniadakis2017SQG} and consider the initial buyancy profile \
\[
	\theta_0(x_1,x_2) = e^{-(x_1-\pi)^2 - 16 (x_2-\pi)^2 },
\]
which develops into a rotating vertex. 

We set $\textrm{CFL}= 0.4$, see \eqref{e:CFL}, and perform the simulations using two different space resolution corresponding to uniform triangulations $\mathcal T_H$ with 351$\times$351 and  $\mathcal T_h$ with 512$\times$512 vertices.
We also investigate the effect of the residual entropy viscosity parameter $c_{\textrm{EV}}$ in \eqref{eq:dijH} chosen to be either $c_{\textrm{EV}}=0.1$ or $c_{\textrm{EV}}=0.5$.

 The buoyancy  at several time snapshots is provided in Figure~\ref{fig:vortex}. 
 The columns of Figure~\ref{fig:vortex} correspond to four simulations: first order solution (first column) on $\mathcal T_H$, FCT solutions with $c_\textrm{EV}=0.5$ on $\mathcal T_H$ (second column), $c_\textrm{EV}=0.1$ on $\mathcal T_H$ (third column), and $c_\textrm{EV}=0.1$ on $\mathcal T_h$ (fourth column). We observe the significant improvement in accuracy of the limiting algorithm when comparing with the first order scheme. 
We also remark that the predictions from all the higher order finite element schemes are comparable to the spectral methods used in \cite{Karniadakis2017SQG}.  
 In particular, we see that the vortex grows thin tails which eventually generate small structures (spinning vortices).
  Furthermore, all simulations exhibit a discrete maximum principle property as predicted by Theorem~\ref{t:FCT}, which does not seem to be the case for the simulations provided in \cite{Karniadakis2017SQG}. 
 
\begin{figure}[hbt!]
%  \centering{
     \includegraphics[width=0.23\textwidth]{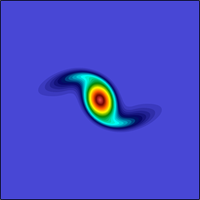}
     \includegraphics[width=0.23\textwidth]{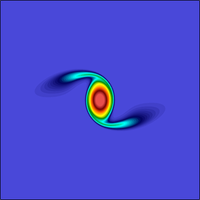}
     \includegraphics[width=0.23\textwidth]{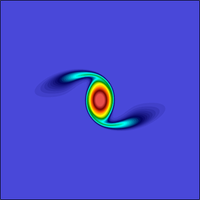}
     \includegraphics[width=0.23\textwidth]{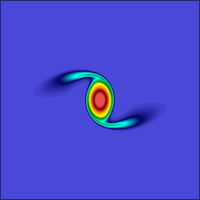}
     \includegraphics[width=0.0348\textwidth]{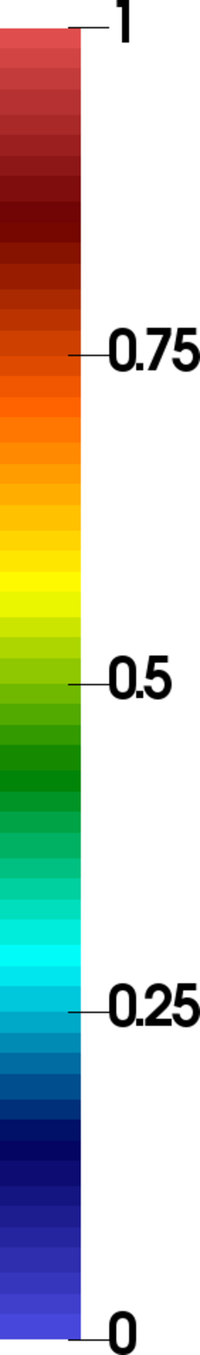}     
      \\
     \includegraphics[width=0.23\textwidth]{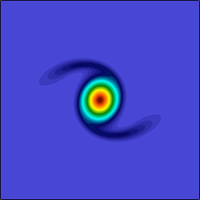}
     \includegraphics[width=0.23\textwidth]{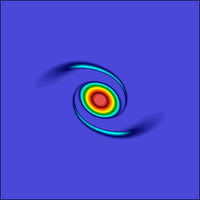}
     \includegraphics[width=0.23\textwidth]{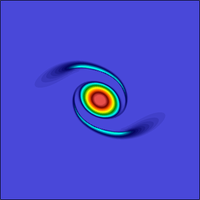}
     \includegraphics[width=0.23\textwidth]{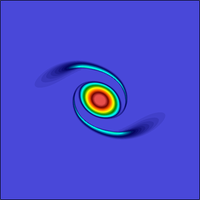}
      \\
     \includegraphics[width=0.23\textwidth]{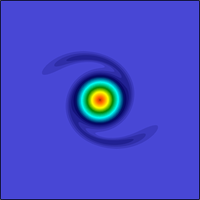}
     \includegraphics[width=0.23\textwidth]{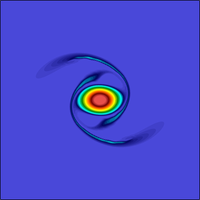}
     \includegraphics[width=0.23\textwidth]{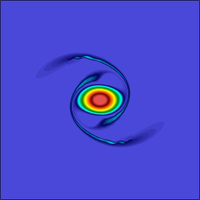}
     \includegraphics[width=0.23\textwidth]{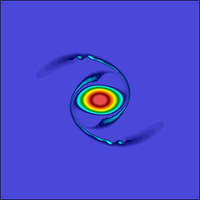}
      \\
     \includegraphics[width=0.23\textwidth]{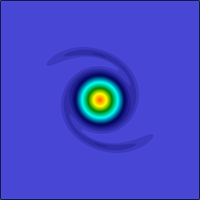}
     \includegraphics[width=0.23\textwidth]{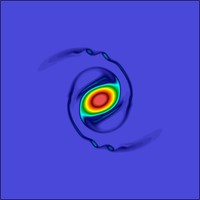}
     \includegraphics[width=0.23\textwidth]{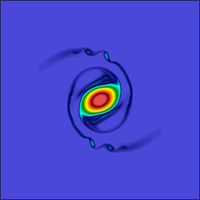}
     \includegraphics[width=0.23\textwidth]{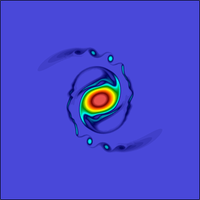}
      \\
     \includegraphics[width=0.23\textwidth]{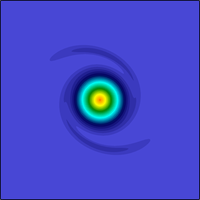}
     \includegraphics[width=0.23\textwidth]{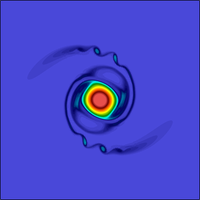}
     \includegraphics[width=0.23\textwidth]{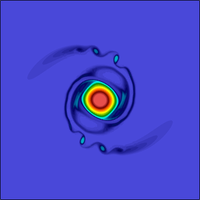}
     \includegraphics[width=0.23\textwidth]{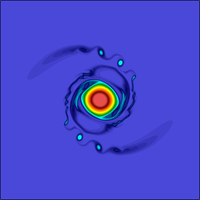}
      \\
%  }
  \caption{
  Single vortex rotation at (by rows) $t=8$, $t=16$,
  $t=26$, $t=35$ and $t=40$; the columns correspond to first order scheme on $\mathcal T_H$ (first), 
  and FCT solution with $c_\textrm{EV}=0.5$ on $\mathcal T_H$ (second), $c_\textrm{EV}=0.1$ on $\mathcal T_H$ (third), and $c_\textrm{EV}=0.1$ on the finer mesh $\mathcal T_h$ (fourth). 
  All simulations satisfy a discrete maximum principle property.
  } 
  \label{fig:vortex}
\end{figure}

In Figure~\ref{fig:vortex:plots}, we report the evolution of the  kinetic energy and  helicity \eqref{eq:kinetic_helicity}, which are conserved at the continuous level. These quantities are not conserved by the FCT scheme due to the presence of the artificial viscosity.
We also report in Figure~\ref{fig:vortex-double:plots} (left) the evolution of $\|\nabla \Theta_k(t)\|_{L^\infty{({\mathbb T^2})}}$  to monitor apparition of singularities.
The norm of the gradient of the solution oscillates when long vortex filaments develop and eventually break down to a small scale new vortices (between $t=10$ and $t=40$). 

\begin{figure}[hbt!]
  \centering{
      \includegraphics[width=0.49\textwidth]{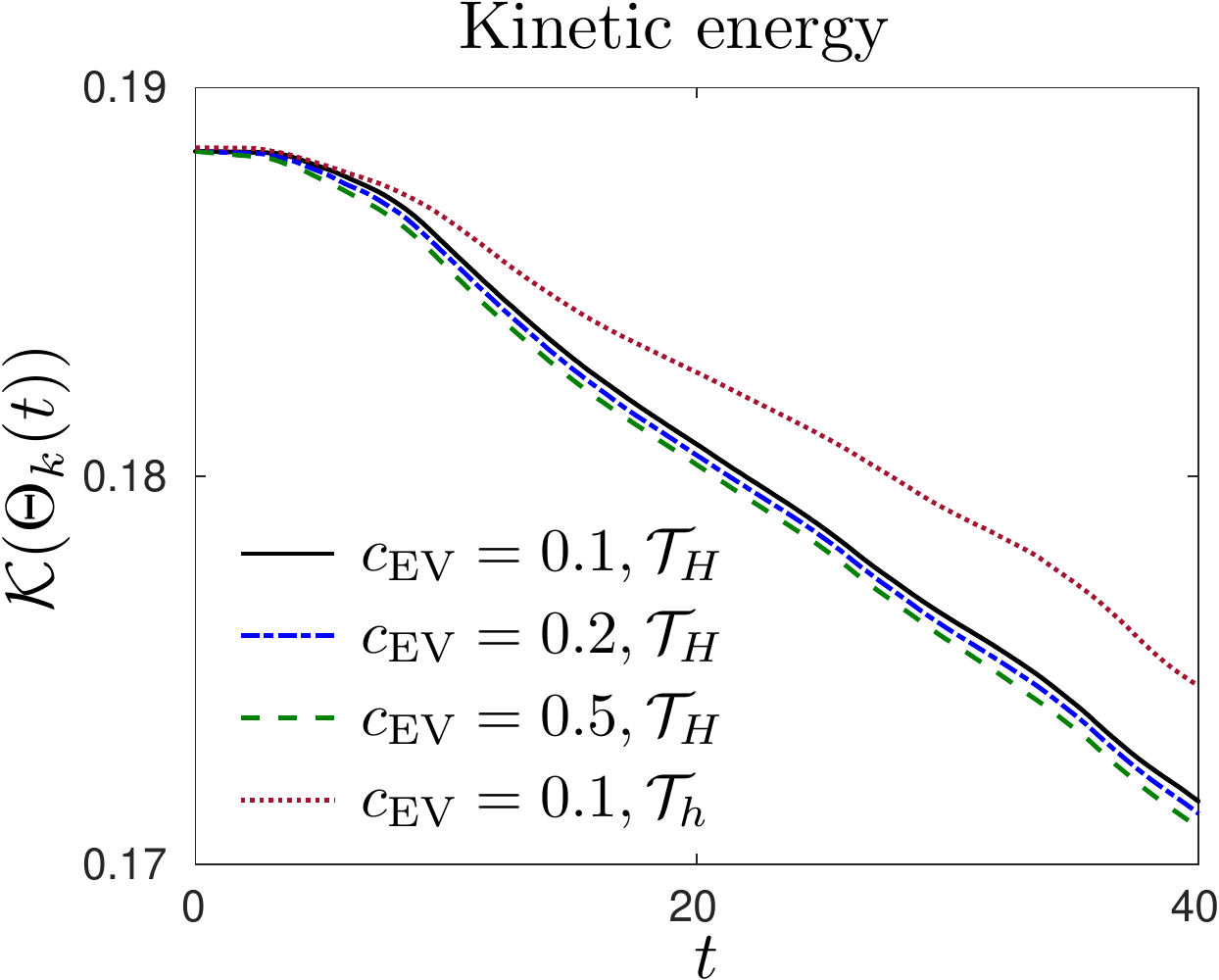}
      \includegraphics[width=0.49\textwidth]{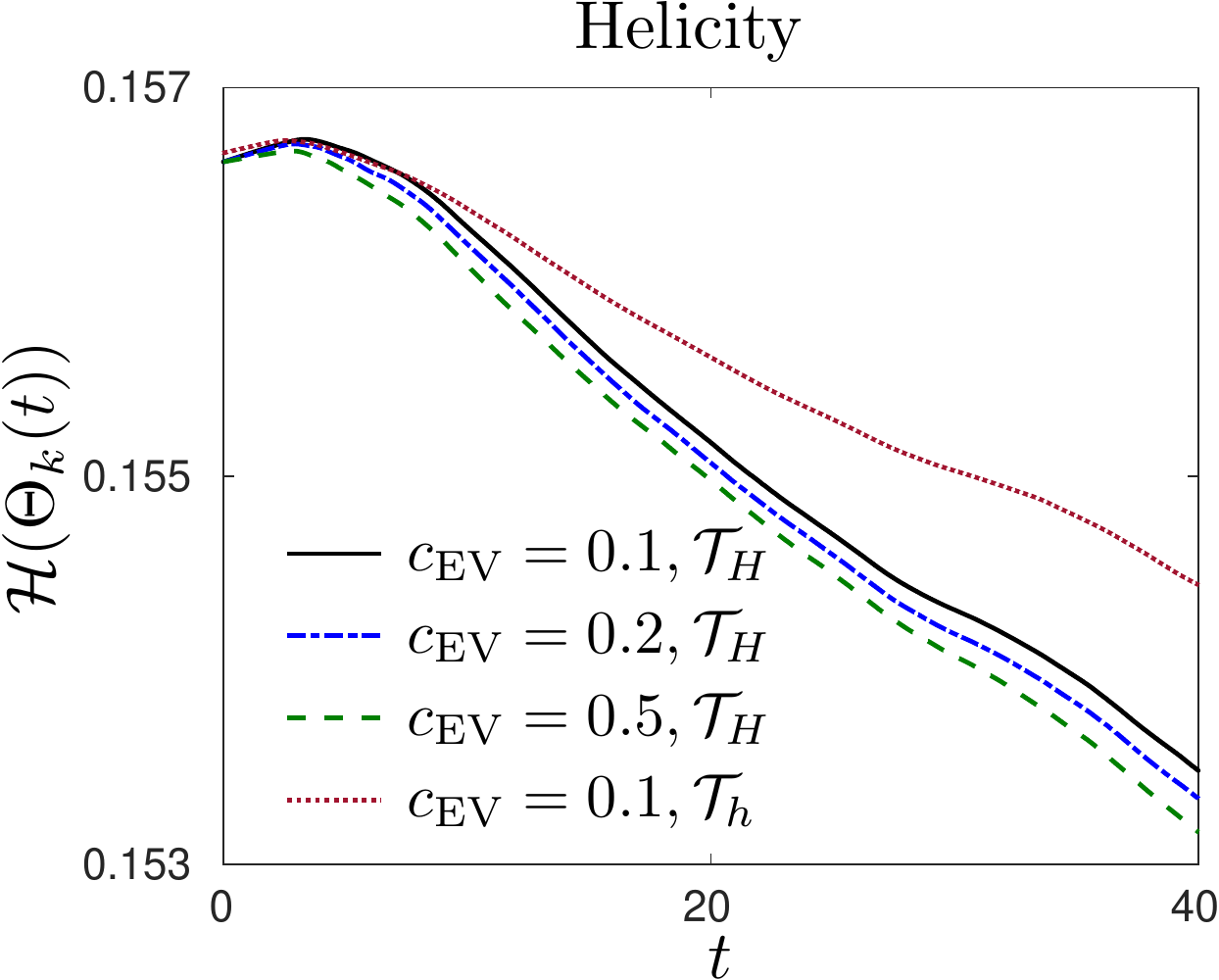}
  }
  \caption{Single vortex rotation: evolution of the kinetic energy and helicity for different nonlinear viscosity parameters $c_{\textrm{EV}}$ and two different triangulations $\mathcal T_H$ (351$\times$351 vertices) and $\mathcal T_h$ (512$\times$512 vertices - indicated ``fine'' in the caption).}
  \label{fig:vortex:plots}
\end{figure}

\begin{figure}[hbt!]
  \centering{
      \includegraphics[width=0.49\textwidth]{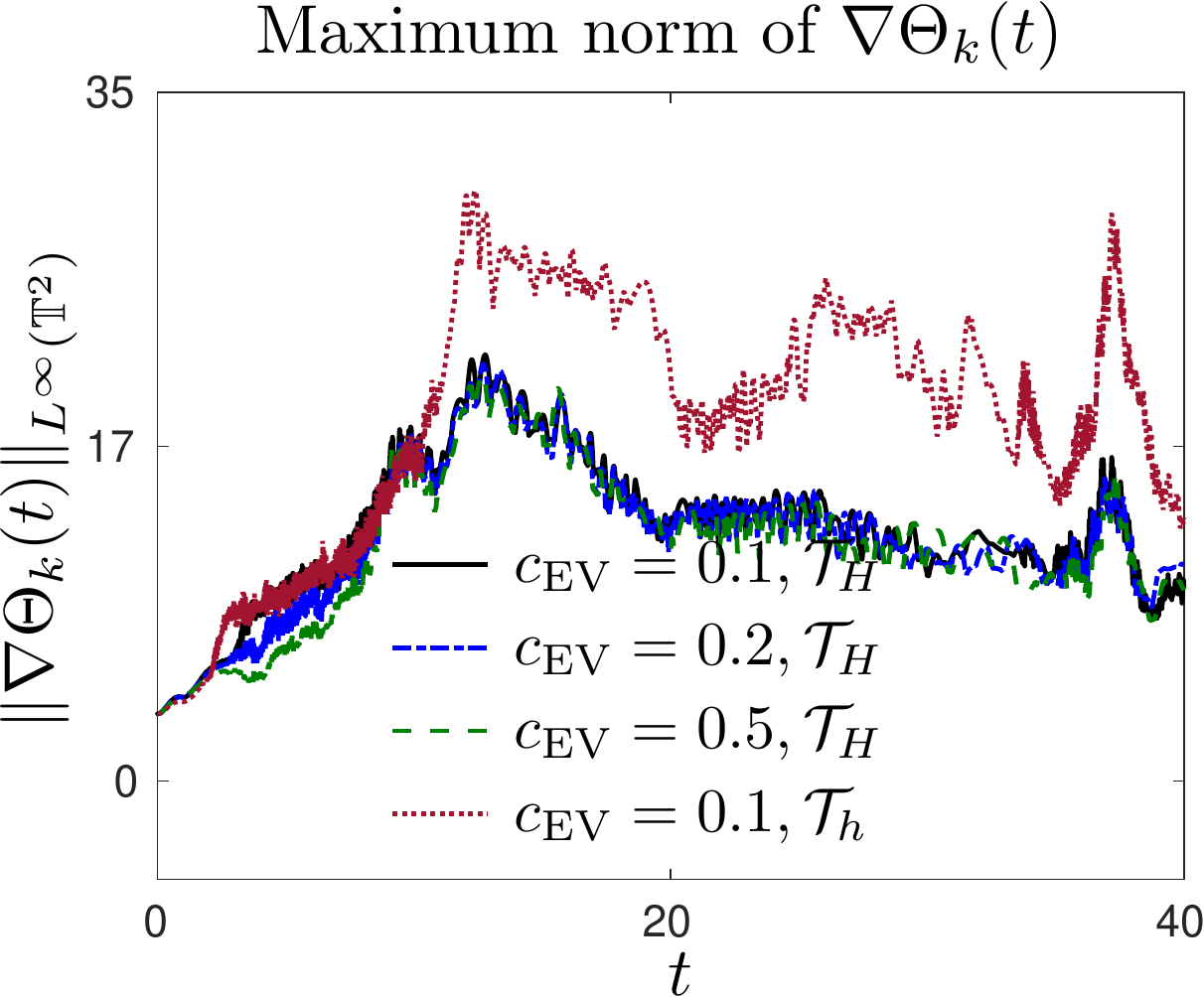}
      \includegraphics[width=0.49\textwidth]{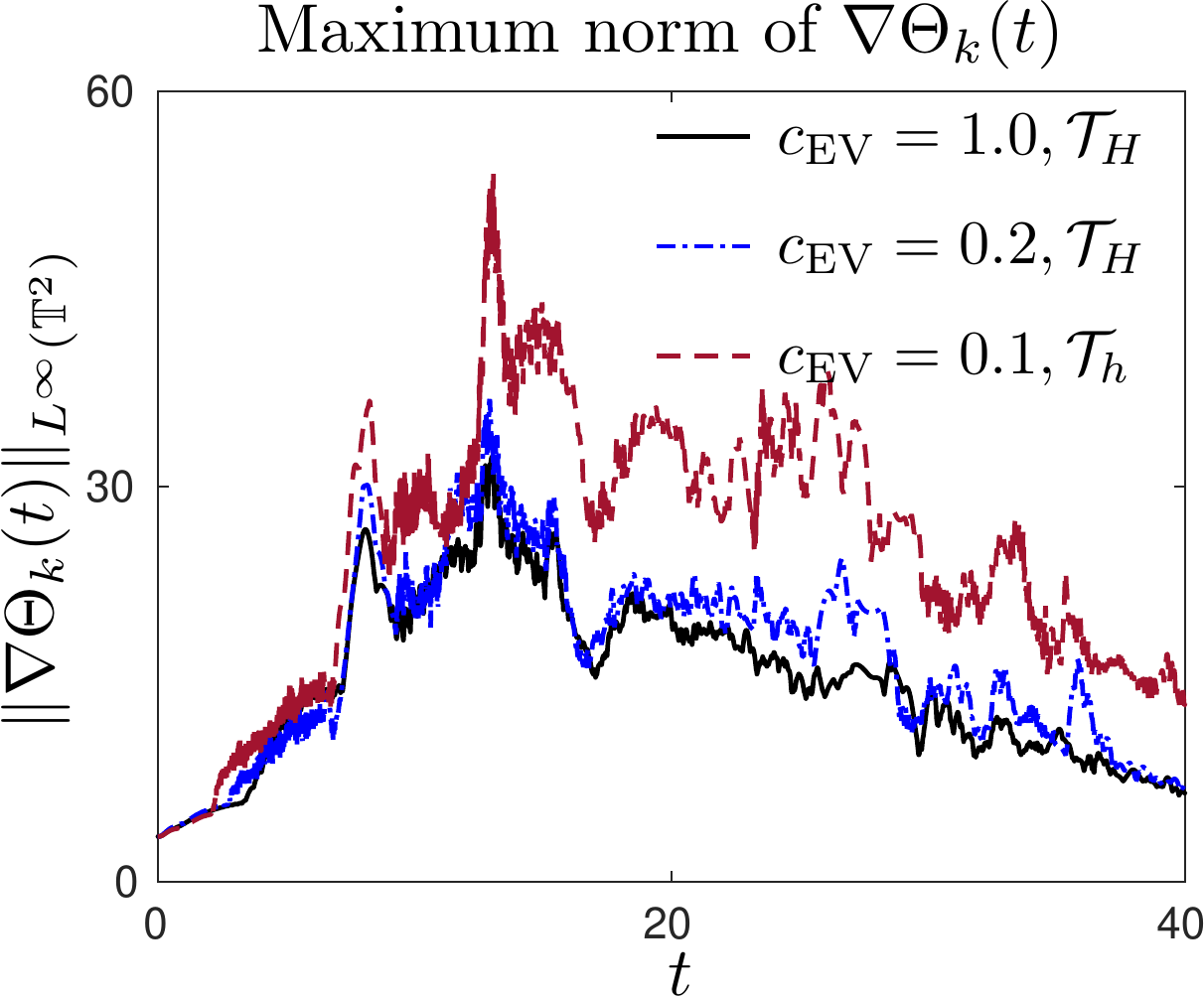}
  }
  \caption{Evolution of $\|\nabla \Theta_k(t)\|_{L^\infty{({\mathbb T^2})}}$ for the (left) single and (right) double vortex rotation for the different nonlinear viscosity parameters $c_{\textrm{EV}}$ and two different triangulations $\mathcal T_H$ (351$\times$351 vertices) and $\mathcal T_h$ (512$\times$512 vertices).}
  \label{fig:vortex-double:plots}
\end{figure}

We now investigate the interaction between two rotating vortices and consider the following initial temperature

\[
	\theta_0(x_1,x_2) = e^{-16(x_1-\pi-\frac12)^2 - (x_2-\pi)^2 } +
	e^{-16(x_1-\pi+\frac12)^2 - (x_2-\pi)^2 }.
\]
Again, we stop the simulation at $T=40$.
As for the single vertex simulation, we select the time step with $\textrm{CFL}= 0.4$ on the coarse mesh $\mathcal T_h$.
The entropy viscosity parameter $c_{\textrm{EV}}$ is set to 0.1.

Snapshots of the buoyancy $\Theta_k^n$ are provided in Figure~\ref{fig:double_vortex}.
We observe the development of a sharp layer between the two vortices around $t=8$. Then this sharp layer reduces over time but the tip of two vortices do not merge. 
We also note the presence of other small scale vortices that develop in time confirming the ability of the method to produce fine details without over-resolving.
Notice that the above mentioned sharp layer does not appear to be a singularity in view of the evolution of $\| \nabla \Theta_k(t) \|_{L^\infty(\mathbb T^2)}$ provided in Figure~\ref{fig:vortex-double:plots} (right). 

\begin{figure}[hbt!]
  \centering{
    \begin{subfigure}[t]{0.31\textwidth}
      \includegraphics[width=\textwidth]{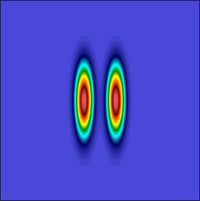}
      \caption{$t=0$}
    \end{subfigure}
    \begin{subfigure}[t]{0.31\textwidth}
      \includegraphics[width=\textwidth]{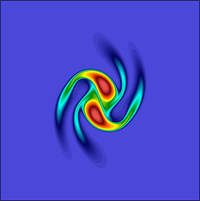}
      \caption{$t=8$}
    \end{subfigure}
    \begin{subfigure}[t]{0.31\textwidth}
      \includegraphics[width=\textwidth]{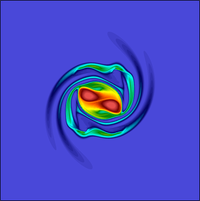}
      \caption{$t=16$}
    \end{subfigure}
    \begin{subfigure}[t]{0.038\textwidth}
      \includegraphics[width=\textwidth]{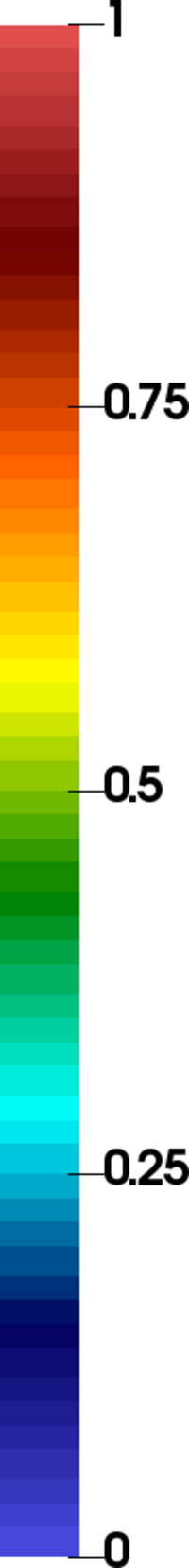}     
    \end{subfigure}
  }
  \centering{
    \begin{subfigure}[t]{0.31\textwidth}
      \includegraphics[width=\textwidth]{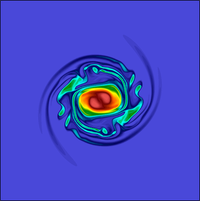}
      \caption{$t=26$}
    \end{subfigure}
    \begin{subfigure}[t]{0.31\textwidth}
      \includegraphics[width=\textwidth]{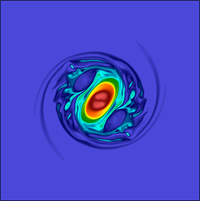}
      \caption{$t=35$}
    \end{subfigure}
    \begin{subfigure}[t]{0.31\textwidth}
      \includegraphics[width=\textwidth]{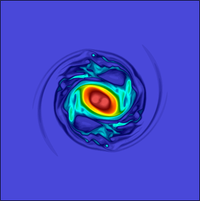}
      \caption{$t=40$}
    \end{subfigure}
    \hspace{0.175in}
  }
  \caption{Snapshots of the double vortex rotation on a $351\times 351$ space resolution and with $\textrm{CFL}=0.4$. 
  A sharp layer develop between the two vortices around $t=8$. Although the intensity of the layer separating the two vortices is reducing over  time, the vortices do not merge.}
  \label{fig:double_vortex}
\end{figure}

\subsection{Viscous SQG with Sharp Transitions}\label{s:viscous_transition}

In this section, we turn our attention to the case $\varkappa>0$ and reproduce the benchmark configuration previously investigated in \cite{constantin1994SQG, ohkitani1997SQG} and \cite{constantin2012SQGsimu, Karniadakis2017SQG}. 
We set $\varkappa$ are interested in approximating the solution to the viscous SQG system with $\varkappa =0.001$. 
Recall that the Ekman pumping number $\varkappa$ is typically small in our physical setting. 
We have already mentioned that singularities can develop only when a saddle structure is present in the initial buoyancy
as for
\[
\theta_0(x_1,x_2) = \sin x_1 \sin x_2 + \cos x_2.
\]

The space discretization consists of 351$\times$351 vertices and we chose a time step so that $\textrm{CFL}= 0.25$.
In addition, the viscosity coefficient $c_{\text{EV}}$ is taken to be  1. 

The evolution of the potential temperature is depicted in Figure~\ref{fig:viscous}. 
The initial data contains two smooth waves that evolve in time and at $t\approx7$ a sharp transition appears between them. A similar sharp transition develops further in other parts of the computational domain.
In addition, we also report in Figure~\ref{fig:viscous:plots} the evolution of $\|\nabla \Theta_k(t)\|_{L^\infty{({\mathbb T^2})}}$.  
As in the above mentioned previous works, we observe that the latter grows when sharp transitions appear and then oscillates.

\begin{figure}[hbt!]
  \centering{
    \begin{subfigure}[t]{0.31\textwidth}
      \includegraphics[width=\textwidth]{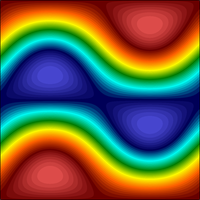}
      \caption{$t=0$}
    \end{subfigure}
    \begin{subfigure}[t]{0.31\textwidth}
      \includegraphics[width=\textwidth]{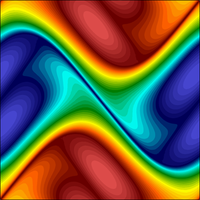}
      \caption{$t=6$}
    \end{subfigure}
    \begin{subfigure}[t]{0.31\textwidth}
      \includegraphics[width=\textwidth]{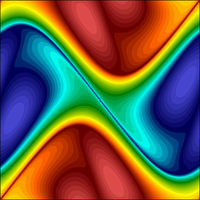}
      \caption{$t=7$}
    \end{subfigure}
    \begin{subfigure}[t]{0.0405\textwidth}
      \includegraphics[width=\textwidth]{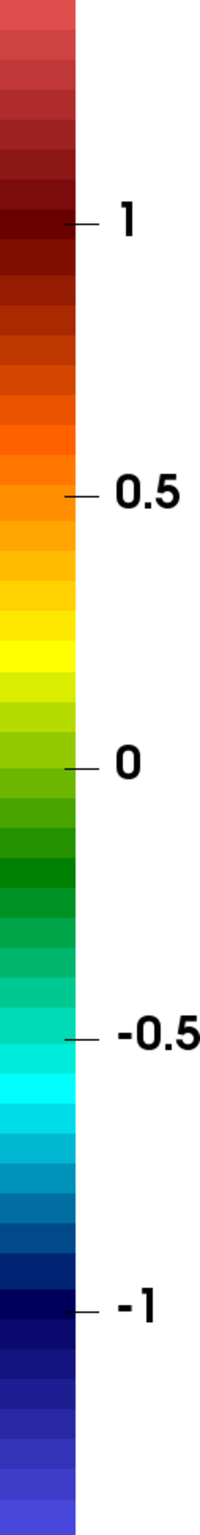}
    \end{subfigure}
  }
  \\
  \centering{
    \begin{subfigure}[t]{0.31\textwidth}
      \includegraphics[width=\textwidth]{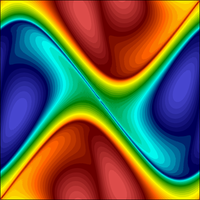}
      \caption{$t=7.5$}
    \end{subfigure}
    \begin{subfigure}[t]{0.31\textwidth}
      \includegraphics[width=\textwidth]{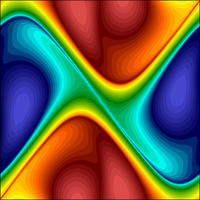}
      \caption{$t=8$}
    \end{subfigure}
    \begin{subfigure}[t]{0.31\textwidth}
      \includegraphics[width=\textwidth]{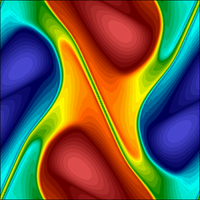}
      \caption{$t=12$}
    \end{subfigure}
    \hspace{0.175in}
  }
  \\
  \centering{
    \begin{subfigure}[t]{0.31\textwidth}
      \includegraphics[width=\textwidth]{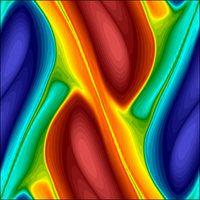}
      \caption{$t=14$}
    \end{subfigure}
    \begin{subfigure}[t]{0.31\textwidth}
      \includegraphics[width=\textwidth]{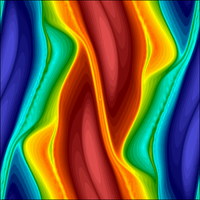}
      \caption{$t=16$}
    \end{subfigure}
    \begin{subfigure}[t]{0.31\textwidth}
      \includegraphics[width=\textwidth]{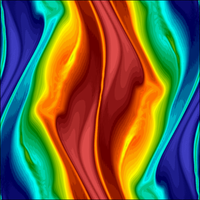}
      \caption{$t=20$}
    \end{subfigure}
    \hspace{0.171in}
  }
  \caption{Snapshots of the potential temperature for the viscous SQG. The mesh consists of $351\!\!\times \!\!351$ vertices, $c_{\text{EV}}=1$ and $\textrm{CFL}= 0.25$. Sharp layers are developing out of the saddle configuration present in the initial data.}
 \label{fig:viscous}
\end{figure}

\begin{figure}[hbt!]
  \centering{
%      \includegraphics[width=0.5\textwidth]{FIGS/NEW/viscous/plots/energy.pdf}\\
%      \vspace{0.2in}
 %     \includegraphics[width=0.5\textwidth]{FIGS/NEW/viscous/plots/helicity.pdf}
      \includegraphics[width=0.49\textwidth]{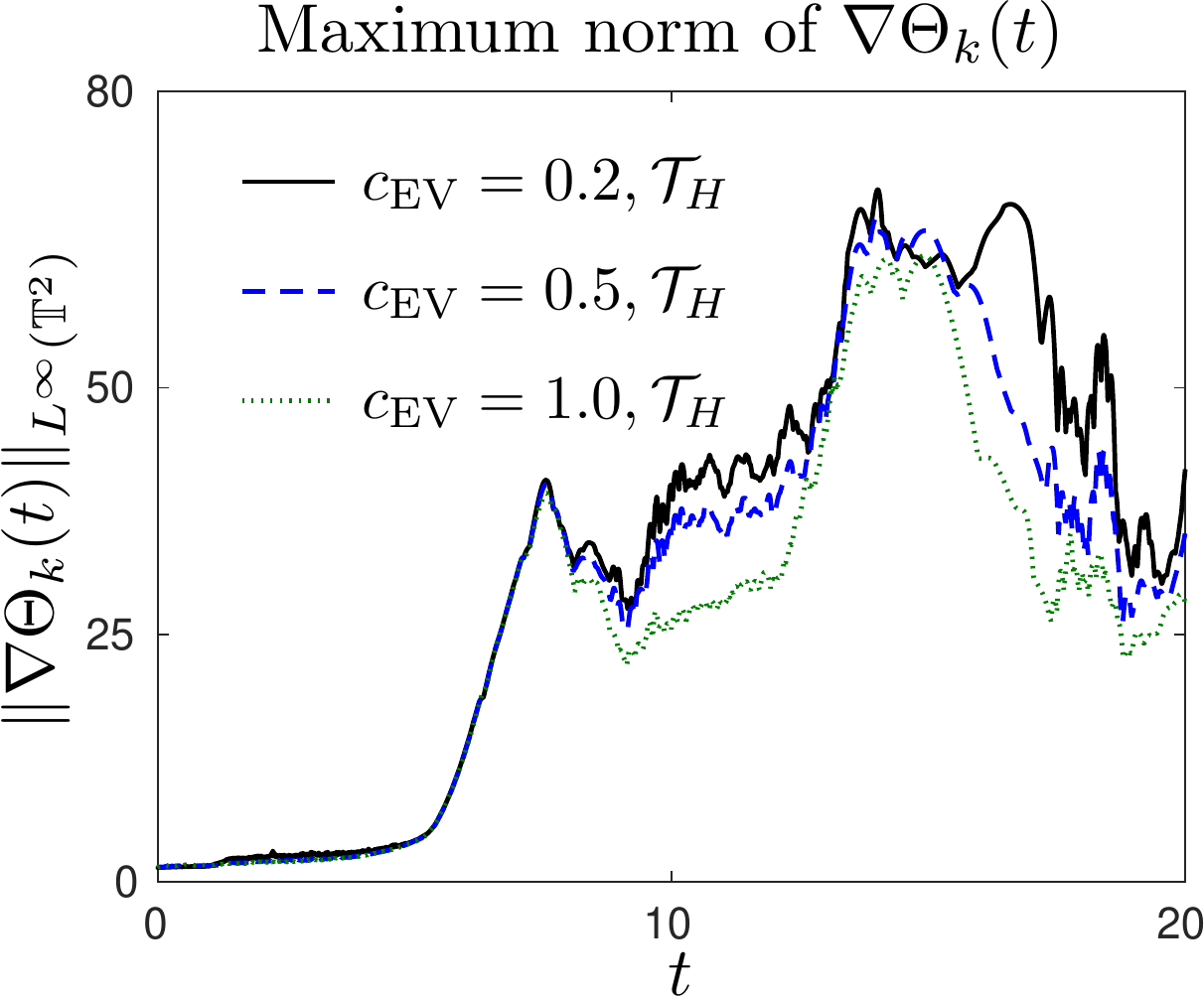}
  }
  \caption{Viscous SQG with sharp layer: Evolution of $\|\nabla \Theta_k(t)\|_{L^\infty{({\mathbb T^2})}}$.  }
  \label{fig:viscous:plots}
\end{figure}

\subsection{Freely Decaying Turbulence and Kolmogorov Energy Cascade}
\label{sec:2d_turbulence}

Experimentally  \cite{Tabeling_et_al_1991} and numerically  \cite{mcwilliams_1984, mcwilliams_1990} it is observed that starting from a noisy initial data to represent incoherent vortices, vortices  appear and merge with other vortices with the same rotation direction to form bigger vortices. This process continuous until only two vertices with opposing rotating velocities are left and decay diffusively.

 We perform a simulation up to time $T=80$ of freely decaying turbulence for the SQG system using a triangulation of the domain $\mathbb T^2$ with vertices of coordinates $(2\pi n/512, 2\pi m/512)$, $n,m=0,...,512$.  
The time step is chosen so that the CFL number is 0.4 and $c_{\textrm{EV}}=0.1$.
Incoherent vortices are represented here by an initial buoyancy whose value at each vertex coordinates is randomly chosen from a uniform partition over $[-10, 10]$, see Figure~\ref{fig:2dturb:sol1} (a).
As expected, vortices emerge and merge with other vortices with the same rotation direction to form bigger vortices as observed in Figure~\ref{fig:2dturb:sol1}. In order to make all small scale vortices visible, we plot the  solution  in Figure~\ref{fig:2dturb:sol1} in Schlieren gray-scale diagram:
\[
  \sigma = \exp\Bigg(
  -10 \frac{|\GRAD \Theta_k^n| - \min_{\polT^2}|\GRAD \Theta_k^n|}{\max_{\polT^2}|\GRAD \Theta_k^n| - \min_{\polT^2}|\GRAD \Theta_k^n|}
  \Bigg).
\]

\begin{figure}[hbt!]
  \centering{
    \begin{subfigure}[t]{0.32\textwidth}
 	   \includegraphics[width=\textwidth]{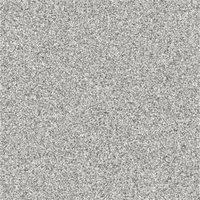}
      \caption{$t=0$}
    \end{subfigure}
    \begin{subfigure}[t]{0.32\textwidth}
 	   \includegraphics[width=\textwidth]{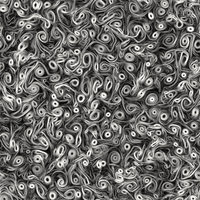}
      \caption{$t=5$}
    \end{subfigure}
    \begin{subfigure}[t]{0.32\textwidth}
 	   \includegraphics[width=\textwidth]{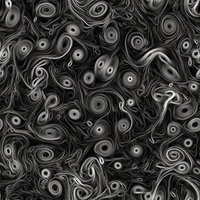}
      \caption{$t=20$}
    \end{subfigure}

    \begin{subfigure}[t]{0.32\textwidth}
 	   \includegraphics[width=\textwidth]{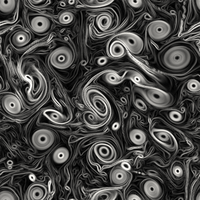}
      \caption{$t=40$}
    \end{subfigure}
    \begin{subfigure}[t]{0.32\textwidth}
 	   \includegraphics[width=\textwidth]{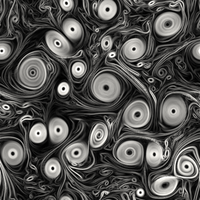}
      \caption{$t=60$}
    \end{subfigure}
    \begin{subfigure}[t]{0.32\textwidth}
 	   \includegraphics[width=\textwidth]{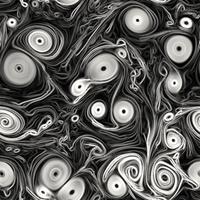}
      \caption{$t=80$}
    \end{subfigure}
    }
    \\
  \caption{Freely decaying turbulence: Schlieren diagram of the potential temperature starting from a white noise initial data.}
  \label{fig:2dturb:sol1}
\end{figure}

Kolmogorov energy cascade describes the energy transfer from larger scale vortices to the smaller ones. 
For isotropic flows like in this setting, it suffices to consider the kinetic energy $\calK(t)$ in \eqref{eq:kinetic_helicity} and determine the energy distribution $\widehat{R}(k,t)$ for $k=1,2,...$ so that
$$
\calK(t) = \sum_{k=0}^\infty \widehat{R}(k,t)dk.
$$

We briefly describe the process and refer to \cite{Pope_2000} for additional details and the Kolmogorov assumptions.
We denote by $R(y,t)$, $y \in \mathbb R$, $t\geq 0$,  the two point correlation function in the first variable
$$
R(y,t) := \frac1{2}\int_{{\mathbb T^2}} \theta( (x_1,x_2) ,t)\theta( (x_1+y,x_2),t)dx_1 dx_2.
$$
Note that because the fluid is assumed to be isotropic,  it suffices to compute the correlation with respect to one variable (the first here) and that
$$
\calK(t) = R(0,t).
$$
The wavelength decomposition readily follows from the Fourier series of $R(y,t)$:
\begin{equation}\label{e:Fourier_exp}
R(y,t)= \sum_{n=-\infty}^\infty  \widehat R_{n}(t)e^{-i ny}
\end{equation}
or, regrouping the terms
$$
R(y,t)= \sum_{k=0}^\infty \sum_{|n|=k}  \widehat R_{n}(t)e^{-i ny}.
$$
Whence, we obtain the desired expression
$$
\calK(t) =  \sum_{k=0}^\infty \widehat R(k,t) \qquad \textrm{with}\qquad 
 \widehat R(k,t):=\sum_{|n|=k}  \widehat R_{n}(t).
$$

In practice, we use discrete Fourier Transform with $N=512$ terms to approximate the Fourier expansion \eqref{e:Fourier_exp}
$$
\widetilde R_m(t) =  \frac{1}{N}\sum_{n=0}^{N-1} R(2\pi m/N,t)e^{-\frac{2\pi i }{N}m},
$$
with $m=0,...,N-1$ so that
$$
\calK(t) \approx  \frac 1 N \sum_{m=0}^{N-1}  | \widetilde R_m(t)|.
$$

We present the energy cascade $m \mapsto |\widetilde R_m(t)|$  for several times in the fully inviscid SQG case in Figure~\ref{fig:2D:E}. 
At large scales, the theoretical prediction of the energy decay $-\frac53$ as in full three-dimensional turbulence was obtained by \cite{held1995surface,Pierrehumbert_et_al_1994}. However inverse cascade of energy typical of two dimensional flows are predicted at small scales thereby leading to an energy decay of $-3$ \cite{Tulloch14690,lapeyre2017surface,Ragone_Badin_2016,capet2008surface}. The left panel of Figure~\ref{fig:2D:E} (left) depicts the energy decay for the SQG simulation. For smaller wave numbers we observe the theoretical $-\frac53$ rate, however, for bigger wavenumbers, the slope becomes steeper. 
We note that similar decays are observed in the viscous case.
The energy decay when $\varkappa=0.001$ and $s=\frac 1 2$ is reported in Figure~\ref{fig:2D:E} (middle).

For comparison, we mention that in the quasi-geostrophic (QG) system, the stream function is computed via the relation
\begin{equation}\label{e:QG}
(-\Delta) \psi = \theta
\end{equation}
instead of \eqref{eq:sqg_frac}.
This system is widely used to mostly study 2D turbulence, see for instance \cite{mcwilliams_1984, mcwilliams_1990, Carnevale_et_al_1991} and references therein.
Our numerical experiments reproduce the expected decay of approximately $-5$ predicted in \cite{mcwilliams_1984, mcwilliams_1990}, see Figure~\ref{fig:2D:E} (right).

\begin{figure}[hbt!]
\centering{
  \includegraphics[width=0.32\textwidth]{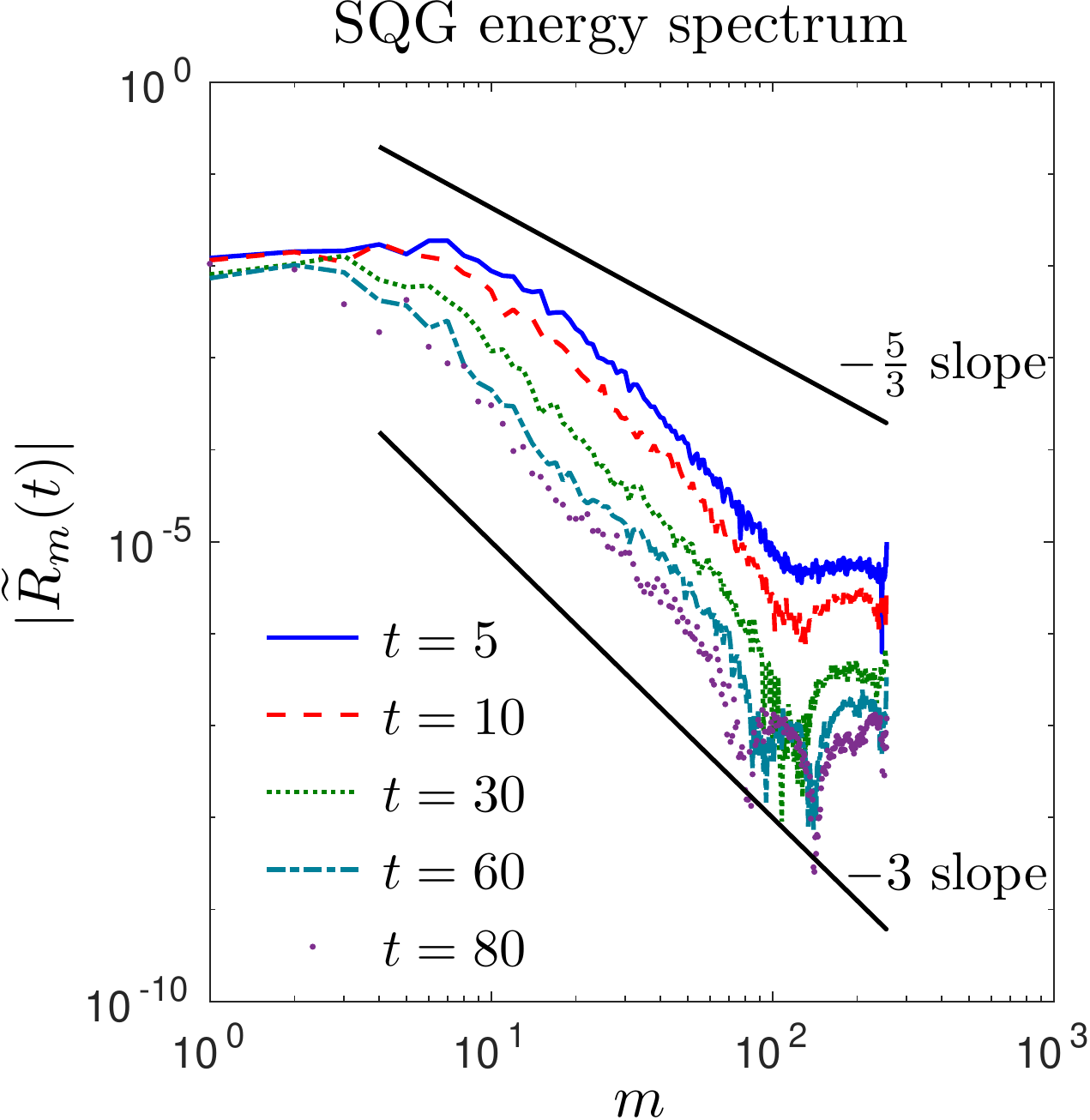}
  \includegraphics[width=0.32\textwidth]{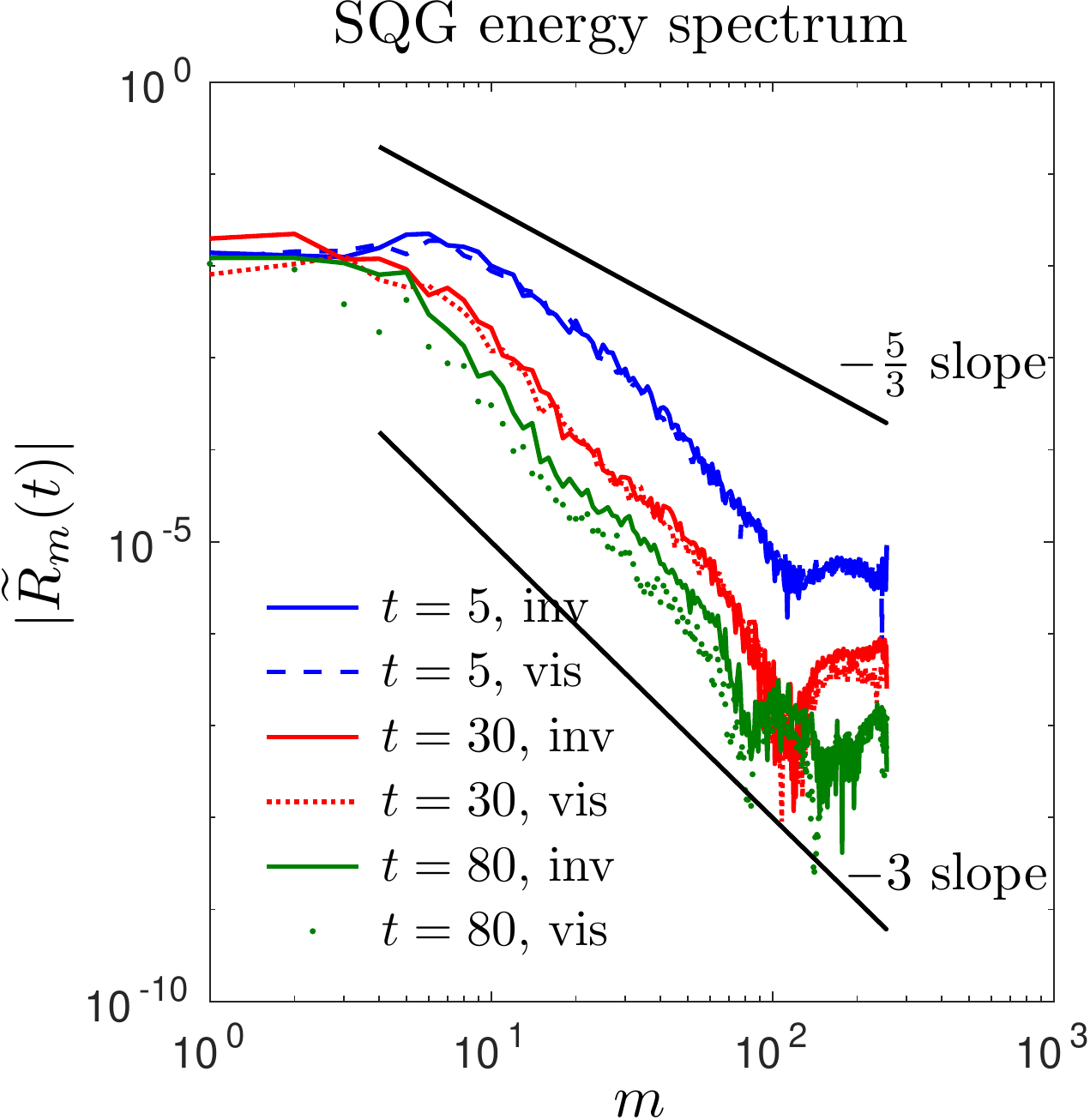}
  \includegraphics[width=0.32\textwidth]{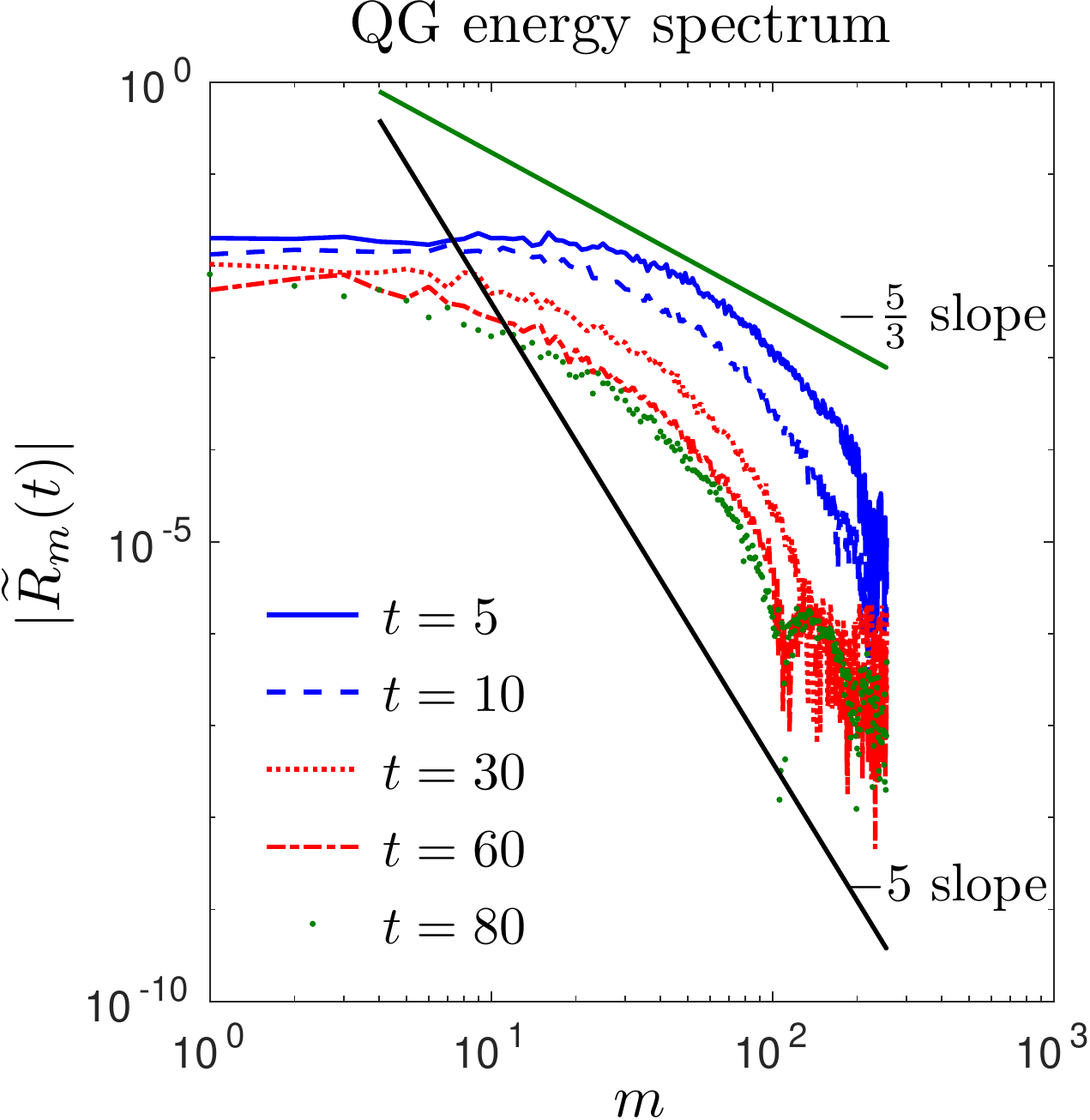}
}
  \caption{2D freely decaying turbulence: the energy spectrum $|\widetilde R_m(t)|$ vs $m$, $m=0,...,256$ in the inviscid case for (left) the inviscid SQG system, (middle) the viscous SQG system and (right) the QG system. The observed decay rate $-5/3$ for small wavenumbers and $-3$ for large wavenumbers for the SQG systems are in accordance with the theoretical predictions. Compare with the QG system which exhibits an energy decay rate of $-5$ instead.}
  \label{fig:2D:E}
\end{figure}

\bibliographystyle{abbrv}
\bibliography{references}
\end{document}